\documentclass{amsart}

\usepackage{color}
\usepackage{amssymb}
\usepackage{amsmath}
\usepackage{hyperref}
\def\BibTeX{{\rm B\kern-.05em{\sc i\kern-.025em b}\kern-.08em
    T\kern-.1667em\lower.7ex\hbox{E}\kern-.125emX}}

\newtheorem{thm}{Theorem}[section]
\newtheorem{lem}[thm]{Lemma}
\newtheorem{prop}[thm]{Proposition}

\newtheorem{rem}[thm]{Remark}
\newtheorem*{claim*}{Claim}
\theoremstyle{definition}

\numberwithin{equation}{section}

\usepackage{bbm}

\providecommand{\bysame}{\leavevmode\hbox to3em{\hrulefill}\thinspace}
\providecommand{\MR}{\relax\ifhmode\unskip\space\fi MR }
\providecommand{\href}[2]{#2}

\begin{document}
\allowdisplaybreaks
\title[Exponential ergodicity of an affine two-factor model] % running head version
{Exponential ergodicity of an affine two-factor model based on the $\alpha$-root process}

\author{Peng Jin, Jonas Kremer and Barbara R\"udiger}
\address[Peng Jin]{Fakult\"at f\"ur Mathematik und Naturwissenschaften\\
        Bergische Universit\"at Wuppertal\\
        42119 Wuppertal, Germany}
%% Note the doubled @@:
\email[Peng Jin]{jin@uni-wuppertal.de}
\address[Jonas Kremer]{Fakult\"at f\"ur Mathematik und Naturwissenschaften\\
        Bergische Universit\"at Wuppertal\\
        42119 Wuppertal, Germany}
%% Note the doubled @@:
\email[Jonas Kremer]{j.kremer@uni-wuppertal.de}
\address[Barbara R\"udiger]{Fakult\"at f\"ur Mathematik und Naturwissenschaften\\
        Bergische Universit\"at Wuppertal\\
        42119 Wuppertal, Germany}
%% Note the doubled @@:
\email[Barbara R\"udiger]{ruediger@uni-wuppertal.de}

\date{\today}



\maketitle

\begin{abstract}
We study an affine two-factor model introduced by Barczy \textit{et al.} (2014).
One component of this two-dimensional model is the so-called $\alpha$-root process, which generalizes the well known CIR process. In this paper, we
show that this affine two-factor model is exponentially ergodic when
$\alpha\in(1,2)$.

\end{abstract}

\section{Introduction}
\label{intro}
In this paper, we study a two-dimensional affine process $(Y,X):=(Y_{t},X_{t})_{t\geqslant0}$ determined by the following stochastic differential equation
\begin{equation}
\begin{cases}
 \mathrm{d}Y_{t}=(a-bY_{t})\mathrm{d}t+\sqrt[\alpha]{Y_{t-}}\mathrm{d}L_{t}, & t\geqslant0,  \quad Y_0\geqslant 0 \quad \mbox{a.s.},    \\
\mathrm{d}X_{t}=(m-\theta X_{t})\mathrm{d}t+\sqrt{Y_{t}}\mathrm{d}B_{t}, & t\geqslant0,
\end{cases}\label{eq:SDE}
\end{equation}
where $a>0,b>0,\theta,m\in\mathbb{R}$, $\alpha\in(1,2)$, $(L_{t})_{t\geqslant0}$
is a spectrally positive $\alpha$-stable L\'evy process with the L\'evy
measure $C_{\alpha}z^{-1-\alpha}\mathbbm{1}_{\lbrace z>0\rbrace}\mathrm{d}z$, with $C_{\alpha}:=\left(\alpha\Gamma(-\alpha)\right)^{-1}$, and $(B_{t})_{t\geqslant0}$
is an independent standard Brownian motion. Note that if $(Y_0,X_0)$ is independent of $(L_t,B_t)_{t\geqslant0}$, then the existence
and uniqueness of a strong solution to the SDE \eqref{eq:SDE} follow from \cite[Theorem 2.1]{MR3254346}.

The process $(Y_{t},X_{t})_{t\geqslant0}$ given by \eqref{eq:SDE}
has been introduced by Barczy \textit{et al.} in \cite{MR3254346}. There,
it was proved that $(Y_{t},X_{t})_{t\geqslant0}$ belongs to the class
of regular affine processes (with state space $\mathbb{R}_{\geqslant0}\times\mathbb{R}$). The process $Y$ is the so-called $\alpha$-root
process (sometimes referred as the stable CIR process, shorted SCIR, see \cite{MR3343292}) and is also an affine process (with state space $\mathbb{R}_{\geqslant0}$). It can be considered as an
extension of the CIR process. The general theory of affine processes
on the canonical state space $\mathbb{R}_{\geqslant0}^{m}\times\mathbb{R}^{n}$
was initiated by Duffie \textit{et al.}~\cite{MR1793362} and further developed in \cite{MR1994043}. An affine process
on $\mathbb{R}_{\geqslant0}^{m}\times\mathbb{R}^{n}$
is a continuous-time Markov process taking values in
$\mathbb{R}_{\geqslant0}^{m}\times\mathbb{R}^{n}$, whose log-characteristic
function depends in an affine way on the initial state vector of the
process, i.e. the log-characteristic function is linear with respect to
the initial state vector. Affine processes are particularly important
in financial mathematics because of their computational tractability.
For example, the models of Cox \textit{et al.}~\cite{MR785475}, Heston \cite{Heston}
and Vasicek \cite{MR3235239} are all based on affine processes.

An important issue for the application of affine processes is the
calibration of their parameters. This has been investigated for some
well known affine models, see e.g. \cite{MR1455180,MR1614256,MR2995525,MR3216637,MR3452993}. To study the asymptotic properties of estimators of the parameters,
a comprehension of the long-time behavior of the underlying affine processes is very often required. This is one of the reasons why the stationary, ergodic and recurrent properties
of affine processes have recently attracted many investigations, see e.g. \cite{MR2922631,MR2390186,MR3254346,MR3343292,MR3264444,MR3167406,MR3437080,MR3451177},  and many
others.

Concerning the two-factor model defined
in \eqref{eq:SDE}, it was shown in \cite{MR3254346} that $(Y_{t},X_{t})_{t\geqslant0}$
has a stationary distribution. Using the same argument as in \cite[p.~80]{MR2779872}, it can easily  be seen that the stationary distribution for $(Y_{t},X_{t})_{t\geqslant0}$ is actually unique. If one allows $\alpha=2$ and replaces
$(L_{t})_{t\geqslant0}$ in (\ref{eq:SDE}) by a standard Brownian
motion $(W_{t})_{t\geqslant0}$ (independent of $(B_{t})_{t\geqslant0}$),
then the process $Y$ becomes the CIR process; in this case, the ergodicity
of $(Y_{t},X_{t})_{t\geqslant0}$ has been proved in \cite{MR3254346}. However, the ergodicity of  $(Y_{t},X_{t})_{t\geqslant0}$ in the case $1<\alpha<2$ is still not known.

In this work we study the ergodicity problem for the two-factor model in \eqref{eq:SDE} when $1<\alpha<2$. As our main result (see Theorem \ref{thm:ergodicity_of_y_x} below),
we show that $(Y_{t},X_{t})_{t\geqslant0}$ in \eqref{eq:SDE} is exponentially ergodic if $\alpha\in(1,2)$, complementing the
results in \cite{MR3254346}. Our approach is very close to that of
\cite{MR3451177}. The first step is to show the existence of positive transition densities
of the $\alpha$-root process $Y_t$. To achieve this, we calculate explicitly the Laplace transform of $Y_t$. Through a careful analysis of the decay rate of the Laplace transform of $Y_t$ at infinity, we manage to show the positivity of the density function of $Y_t$ using the inverse Fourier transform.   In the second step,  we  construct a Foster-Lyapunov function
for the process $(Y_{t},X_{t})_{t\geqslant0}$. Using the general
theory in \cite{MR1174380,MR1234294,MR1234295} on the ergodicity of Markov processes, we are then able to obtain the exponential
ergodicity of the process $(Y_{t},X_{t})_{t\geqslant0}$ in \eqref{eq:SDE}.

Finally, we remark that the exponential ergodicity for a large class of affine processes on $\mathbb{R}_{\geqslant0}$, including the $\alpha$-root process $(Y_{t})_{t\geqslant0}$, has been derived in \cite{MR3343292} by a coupling method. We don't know if a similar coupling argument would work for the two-dimensional affine process $(Y_{t},X_{t})_{t\geqslant0}$ in \eqref{eq:SDE}.

The rest of the paper is organized as follows. In Section 2 we recall some basic facts on the process $(Y_{t},X_{t})_{t\geqslant0}$. In Section 3 we derive the Laplace transform of the $\alpha$-root process $Y$.  In Section 4 we prove that the $\alpha$-root process $Y$ possesses positive transition densities. In Section 5 we construct a Foster-Lyapunov function
for the process $(Y_{t},X_{t})_{t\geqslant0}$. In Section 6 we show that the process $(Y_{t},X_{t})_{t\geqslant0}$ is exponentially ergodic.

\section{Preliminaries}

\label{prelim}

In this section we recall some key facts on the affine process $(Y,X):=(Y_{t},X_{t})_{t\geqslant0}$
defined by the equation \eqref{eq:SDE}, mainly due to \cite{MR3254346}.

Let $\mathbb{N}$, $\mathbb{Z}_{\geqslant0}$, $\mathbb{R}$, $\mathbb{R}_{\geqslant0}$
and $\mathbb{R}_{>0}$ denote the sets of positive integers, non-negative
integers, real numbers, non-negative real numbers and strictly positive
real numbers, respectively. Let $\mathbb{C}$ be the set of complex
numbers. For $z\in\mathbb{C}\setminus\{0\}$ we denote by Arg$(z)$
the principal value of its argument and by $\bar{z}$ its conjugate.
We define the following subsets of $\mathbb{C}$:
\begin{align*}
\mathcal{U}_{-}:=\left\lbrace u\in\mathbb{C}\thinspace:\thinspace\mathrm{Re}\thinspace u\leqslant0\right\rbrace , & \quad\mathcal{U}_{+}:=\left\lbrace u\in\mathbb{C}\thinspace:\thinspace\mathrm{Re}\thinspace u\geqslant0\right\rbrace ,\\
\mathcal{U}_{-}^{\mathrm{o}}:=\left\lbrace u\in\mathbb{C}\thinspace:\thinspace\mathrm{Re}\thinspace u<0\right\rbrace , & \quad\mathcal{U}_{+}^{\mathrm{o}}:=\left\lbrace u\in\mathbb{C}\thinspace:\thinspace\mathrm{Re}\thinspace u>0\right\rbrace ,
\end{align*}
and
\[
\mathcal{O}:=\mathbb{C}\setminus\lbrace-x:\thinspace x\in\mathbb{R}_{\geqslant0}\rbrace.
\]
For $z\in\mathbb{C}\setminus\{0\}$ let Log($z$) be the principal
value of the complex logarithm of $z$, i.e., $\mbox{Log}(z)=\ln(|z|)+i\mbox{Arg}(z)$.
For $\beta\in\mathbb{R}$ define the complex power function $z^{\beta}$
as
\begin{equation}
z^{\beta}:=\exp(\beta\,\mbox{Log}\,z),\quad z\in\mathbb{C}\setminus\{0\}.\label{defi: complex power}
\end{equation}

By $C^{2}(S,\mathbb{R})$, $C_{c}^{2}(S,\mathbb{R})$ and $C_{b}^{2}(S,\mathbb{R})$
we denote the sets of $\mathbb{R}$-valued functions on $S$ that
are twice continuously differentiable, that are twice continuously
differentiable with compact support and that are bounded continuous
with bounded continuous first and second order partial derivatives,
respectively, where the space $S$ can be $\mathbb{R}$, $\mathbb{R}_{\geqslant0}\times\mathbb{R}$
or $\mathbb{R}_{\geqslant0}\times\mathbb{R}_{\geqslant0}\times\mathbb{R}$
in this paper.

We assume that $(\Omega,\mathcal{F},\left(\mathcal{F}_{t}\right)_{t\geqslant0},\mathbb{P})$
is a filtered probability space satisfying the usual conditions, i.e.,
$(\Omega,\mathcal{F},\mathbb{P})$ is complete, the filtration $\left(\mathcal{F}_{t}\right)_{t\geqslant0}$
is right-continuous and $\mathcal{F}_{0}$ contains all $\mathbb{P}$-null
sets in $\mathcal{F}$.

Let $(B_{t})_{t\geqslant0}$ be a standard $\left(\mathcal{F}_{t}\right)_{t\geqslant0}$-Brownian
motion and $(L_{t})_{t\geqslant0}$ be a spectrally positive $\left(\mathcal{F}_{t}\right)_{t\geqslant0}$-L\'evy
process with the L\'evy measure $C_{\alpha}z^{-1-\alpha}\mathbbm{1}_{\lbrace z>0\rbrace}\mathrm{d}z$,
where $1<\alpha<2$. Assume $(B_{t})_{t\geqslant0}$ and $(L_{t})_{t\geqslant0}$
are independent. Note that the characteristic function of $L_{1}$
is given by
\[
\mathbb{E}\left[e^{iuL_{1}}\right]=\exp\left\lbrace \int_{0}^{\infty}\left(e^{iuz}-1-iuz\right)C_{\alpha}z^{-1-\alpha}\mathrm{d}z\right\rbrace ,\quad u\in\mathbb{R}.
\]
Let $N(\mathrm{d}s,\mathrm{d}z)$ be a Poisson random measure on $\mathbb{R}_{>0}^{2}$
with the intensity measure $C_{\alpha}z^{-1-\alpha}\mathbbm{1}_{\lbrace z>0\rbrace}\mathrm{d}s\mathrm{d}z$
and $\hat{N}(\mathrm{d}s,\mathrm{d}z)$ be its compensator. Then the
L\'evy-It\^o representation of $L$ takes the form
\begin{equation}
L_{t}=\gamma t+\int_{0}^{t}\int_{\{|z|<1\}}z\tilde{N}(\mathrm{d}s,\mathrm{d}z)+\int_{0}^{t}\int_{\{|z|\geqslant1\}}zN(\mathrm{d}s,\mathrm{d}z),\quad t\geqslant0,\label{eq: Levy ito decom for L_t}
\end{equation}
where $\gamma:=-\mathbb{E}\left[\int_{0}^{1}\int_{\{|z|\geqslant1\}}zN(\mathrm{d}s,\mathrm{d}z)\right]$
and $\tilde{N}(\mathrm{d}s,\mathrm{d}z):=N(\mathrm{d}s,\mathrm{d}z)-\hat{N}(\mathrm{d}s,\mathrm{d}z)$
is the compensated Poisson random measure on $\mathbb{R}_{>0}^{2}$
that corresponds to $N(\mathrm{d}s,\mathrm{d}z)$. We remark that
$\gamma t=\int_{0}^{t}\int_{\lbrace\vert z\vert\geqslant1\rbrace}z\hat{N}(\mathrm{d}s,\mathrm{d}z)$
and
\[
\int_{0}^{t}\int_{\lbrace\vert z\vert\geqslant1\rbrace}zN(\mathrm{d}s,\mathrm{d}z)-\gamma t,\quad t\geqslant0,
\]
is thus a martingale with respect to the filtration $(\mathcal{F}_{t})_{t\geqslant0}$.
It follows from \cite[Theorem 2.1]{MR3254346} that if $(Y_{0},X_{0})$
is independent of $(L_{t},B_{t})_{t\geqslant0}$, then there is a
unique strong solution $(Y_{t},X_{t})_{t\geqslant0}$ of the stochastic
differential equation \eqref{eq:SDE} with
\[
Y_{t}=e^{-bt}\left(Y_{0}+a\int_{0}^{t}e^{bs}\mathrm{d}s+\int_{0}^{t}e^{bs}\sqrt[\alpha]{Y_{s-}}\mathrm{d}L_{s}\right),
\]
and
\[
X_{t}=e^{-\theta t}\left(X_{0}+m\int_{0}^{t}e^{\theta s}\mathrm{d}s+\int_{0}^{t}e^{\theta s}\sqrt{Y_{s}}\mathrm{d}B_{s}\right)
\]
for all $t\geqslant0$. Moreover, $(Y_{t},X_{t})_{t\geqslant0}$ is
a regular affine process, and the infinitesimal generator $\mathcal{A}$
of $(Y,X)$ is given by
\begin{align}
(\mathcal{A}f)(y,x) & =(a-by)\tfrac{\partial}{\partial y}f(y,x)+(m-\theta x)\tfrac{\partial}{\partial x}f(y,x)+\tfrac{1}{2}y\tfrac{\partial^{2}}{\partial x^{2}}f(y,x)\nonumber \\
 & \quad+y\int_{0}^{\infty}\left(f(y+z,x)-f(y,x)-z\tfrac{\partial}{\partial y}f(y,x)\right)C_{\alpha}z^{-1-\alpha}\mathrm{d}z,
\label{eq:generator_of_y_x}
\end{align}
where $(y,x)\in\mathbb{R}_{\geqslant0}\times\mathbb{R}$ and $f\in C_{c}^{2}(\mathbb{R}_{\geqslant0}\times\mathbb{R},\mathbb{R})$.

\section{Laplace transform of the $\alpha$-root process $Y$}
\label{affine structure}

In this section we study the $\alpha$-root process $(Y_t)_{t \geqslant 0}$
defined by
\begin{equation}\label{eq:sde_y}
\mathrm{d}Y_t = (a-bY_t) \mathrm{d}t + \sqrt[\alpha]{Y_{t-}}\mathrm{d}L_t,
\quad t\geqslant 0,  \quad Y_0\geqslant 0 \quad \mbox{a.s.},
\end{equation}
where $a \geqslant 0$, $b >0$, $\alpha \in (1,2)$, $(L_t)_{t
\geqslant 0}$ is a spectrally positive $\alpha$-stable L\'evy process with
the L\'evy measure $C_{\alpha}z^{-1-\alpha} \mathbbm{1}_{\lbrace z >0
\rbrace}\mathrm{d}z$. Without any further specification, we always assume that $Y_0$ is independent of  $(L_t)_{t \geqslant 0}$.

We remark that we have allowed $a=0$ in \eqref{eq:sde_y}, which is different as in \eqref{eq:SDE}. In this case, the SDE \eqref{eq:sde_y} turns into
\begin{equation}\label{eq:sde_y_with_a=0}
\mathrm{d}Y_t = -bY_t \mathrm{d}t + \sqrt[\alpha]{Y_{t-}}\mathrm{d}L_t,
\quad t\geqslant 0,   \quad Y_0\geqslant 0 \quad \mbox{a.s.},
\end{equation}
and, by \cite[Theorem 6.2 and Corollary 6.3]
{MR2584896}, a unique strong solution of \eqref{eq:sde_y_with_a=0} also exists.  The $\alpha$-root process $Y$ is thus well-defined for all $a \geqslant 0$. From now on and till the end of this section, we assume temporally that $a \geqslant 0$.

The solution of the stochastic differential equation \eqref{eq:sde_y}
depends obviously  on its initial value $Y_0$. From now on, we denote by $(Y_t^{y})_{t \geqslant 0}$ the $\alpha$-root
process starting  from a constant initial value $y\in\mathbb{R}
_{\geqslant 0}$, i.e., $(Y_t^{y})_{t \geqslant 0}$ satisfies
\begin{equation}\label{defi: Y^y_t}
\mathrm{d}Y^y_t = (a-bY^y_t) \mathrm{d}t + \sqrt[\alpha]{Y^y_{t-}}\mathrm{d}L_t,
\quad t\geqslant 0,  \quad Y^y_0=y.
\end{equation}

Since the $\alpha$-root process is an affine process, the
corresponding characteristic functions of $(Y_t^{y})_{t \geqslant 0}$ are of
affine form, namely,
\begin{equation}\label{eq:chara_Y}
\mathbb{E}\left[ e^{u Y_t^{y}} \right] =
e^{\phi(t,u) + y \psi(t,u)}, \quad u \in \mathcal{U} _-.
\end{equation}
The functions $\phi$ and $\psi$ in turn are given  as solutions of the
generalized Riccati equations
\begin{equation}\label{eq:riccati}
\begin{cases}
\tfrac{\partial}{\partial t} \phi(t,u) = F\left( \psi(t,u) \right), &
\phi(0,u)=0, \\
\tfrac{\partial}{\partial t} \psi(t,u) = R\left( \psi(t,u) \right), &
\psi(0,u)=u \in \mathcal{U}_-,
\end{cases}
\end{equation}
with
\[
F(u) = a u \quad \text{and} \quad R(u) = -b u + \tfrac{(-u)^{\alpha}}
{\alpha},
\]
see \cite[Theorem 3.1]{MR3254346}. An equivalent equation for $\psi$ (see
\eqref{eq:riccati_y} below) was studied in \cite[Theorem 3.1]{MR3254346}. In
particular, it follows from \cite[Theorem 3.1]{MR3254346} that the equation
\eqref{eq:riccati_y} below has a unique solution. However, the explicit form of
the solution to \eqref{eq:riccati_y} has not been derived in
\cite{MR3254346}. In order to study the transition densities of the $\alpha$-root process, we will find the explicit form of the solution
to  \eqref{eq:riccati_y} in the following theorem.

\begin{prop}\label{thm3.1}
Let $a  \geqslant 0 $, $b >0$. Define $v_t(\lambda):=-\psi(t,-\lambda)  $,  $\lambda \in \mathbb{R}_{> 0}$. Then $v_t(\lambda)$ solves the differential equation
\begin{equation}\label{eq:riccati_y}
\begin{cases}
\tfrac{\partial}{\partial t} v_t(\lambda) = -b v_t(\lambda) -		
\frac{1}{\alpha}
\left( v_t(\lambda) \right)^{\alpha}, & t\geqslant 0,\\
\quad v_0(\lambda) = \lambda,
\end{cases}
\end{equation}
where $\lambda \in \mathbb{R}_{>0}$. The unique solution to \eqref{eq:riccati_y} is given by
\begin{equation}\label{defi:v_t}
v_t(\lambda) = \left( \left( \frac{1}{\alpha b} + \lambda^{(1-
\alpha)} \right) e^{b (\alpha -1) t} - \frac{1}{\alpha b}
\right)^{\frac{1}{1-\alpha}}, \quad t \geqslant 0.
\end{equation}
Moreover, the Laplace transform of
$Y_t^{y}$ is given by
\begin{align}
\mathbb{E}\left[e^{-\lambda Y_t^{y}}\right] &=\exp\left\lbrace -a \int_0^t v_s(\lambda)\mathrm{d}s - y
v_t(\lambda) \right\rbrace\nonumber \\
&=\exp\left\lbrace -a \int_0^t \left( \left( \frac{1}{\alpha b} +
\lambda^{(1-\alpha)} \right) e^{b (\alpha -1) s} - \frac{1}{\alpha b}
\right)^{\frac{1}{1-\alpha}} \mathrm{d}s \right.\nonumber \\
&\quad \quad \quad \quad \quad \quad \left. - y \left( \left(
\frac{1}{\alpha b} + \lambda^{(1-\alpha)} \right) e^{b (\alpha -1) t} -
\frac{1}{\alpha b}\right)^{\frac{1}{1-\alpha}} \right\rbrace\label{eq:charfunc}
\end{align}
for all $t \geqslant 0$ and  $\lambda \in \mathbb{R}_{> 0}$.
\end{prop}

\begin{proof}
The equation \eqref{eq:riccati_y} is a Bernoulli differential equation which can be transformed into a linear differential equation through a change of variables. More precisely, if we write $u_t(\lambda)
:= \left( v_t(\lambda)\right)^{1-\alpha}$, then
\begin{align}
\tfrac{\partial}{\partial t}u_t(\lambda)
&=
(1-\alpha) \left(v_t(\lambda)\right)^{-\alpha} \tfrac{\partial}{\partial t}
v_t(\lambda)\nonumber \\
&=
(1-\alpha) \left(v_t(\lambda)\right)^{-\alpha} \left( -
bv_t(\lambda) - \tfrac{1}{\alpha}\left(v_t(\lambda)\right)^{\alpha} \right)
\nonumber\\
&=
b (\alpha-1) u_t(\lambda) + \left( 1-\alpha^{-1}\right)\label{eq:sde_yt}
\end{align}
and $u_0(\lambda) = \left(v_0(\lambda)\right)^{1-\alpha} =
\lambda^{1-\alpha}$. By solving  \eqref{eq:sde_yt}, we obtain
\[
u_t(\lambda) = \left( \frac{1}{\alpha b} + \lambda{1-\alpha} \right)
e^{b (\alpha-1) t} - \frac{1}{\alpha b},
\]
which leads to
\[
v_t(\lambda) = \left( \left( \frac{1}{\alpha b} + \lambda^{1-\alpha}
\right) e^{b (\alpha-1) t} - \frac{1}{\alpha b} \right)^{\frac{1}{1-
\alpha}}
\]
for all $t \geqslant 0$ and $\lambda \in \mathbb{R}_{> 0}$. By \eqref{eq:chara_Y} and \eqref{eq:riccati} and noting that $v_t(\lambda)=-\psi(t,-\lambda)$, we get
\begin{align*}
\mathbb{E}\left[ e^{- \lambda Y_t^{y}} \right]
&=\exp\left\lbrace \phi(t,-\lambda) + y \psi(t,-\lambda) \right\rbrace \\
&= \exp\left\lbrace a \int_0^t \psi(s,-\lambda)\mathrm{d}s - y
v_t(\lambda) \right\rbrace \\
&=
\exp\left\lbrace -a \int_0^t v_s(\lambda)\mathrm{d}s - y
v_t(\lambda) \right\rbrace
\end{align*}
for all $t \geqslant 0$ and $\lambda \in \mathbb{R}_{> 0}$.
\end{proof}

Let
\begin{align*}
\varphi_1(t,\lambda,y) &:= \exp\left\lbrace -y \left( \left( \frac{1}{\alpha b} +
\lambda^{(1-\alpha)} \right) e^{b (\alpha -1) t} - \frac{1}{\alpha b}
\right)^{\frac{1}{1-\alpha}} \right\rbrace, \\
\varphi_2(t,\lambda) &:= \exp\left\lbrace -a \int_0^t \left( \left( \frac{1}{\alpha b} +
\lambda^{(1-\alpha)} \right) e^{b (\alpha -1) s} - \frac{1}{\alpha b}
\right)^{\frac{1}{1-\alpha}} \mathrm{d}s \right\rbrace.
\end{align*}
Then
\begin{equation}\label{eq: conv for Y^y_t}
\mathbb{E} \left[e^{-\lambda Y_t^{y}} \right] =\varphi_1(t,\lambda,y) \cdot \varphi_2(t,\lambda).
\end{equation}
Keeping this decomposition of the Laplace transform  of $Y_t^{y}
$ in mind, we take a closer look at the following two special cases:

\subsection{Special case i): $a=0$.}\label{case_i} To avoid abuse of notations, we use  $
(Z_t^{y})_{t\geqslant 0}$ to denote the strong solution of the stochastic differential equation
\[
\mathrm{d}Z_t^{y} = -bZ_t^{y} \mathrm{d}t + \sqrt[\alpha]{
Z_{t-}^{y}}\mathrm{d}L_t, \quad t \geqslant 0, \quad Z^y_0=y\geqslant 0 .
\]
According to \eqref{eq:charfunc}, the corresponding Laplace transform of $Z_t^{y}$ coincides with $\varphi_1(t,\lambda,y)$. Noting that $b >0$,   we get
\begin{equation}
\lim\limits_{\lambda \to \infty} v_t(\lambda)
= \left( \frac{1}{\alpha b} \left( e^{b (\alpha-1) t}-1 \right)
\right)^{\frac{1}{1-\alpha}}=: d >0 \label{defi:d}
\end{equation}
for all $t > 0$. Furthermore,
by dominated convergence theorem, we have
\begin{align}
e^{-y d}&= \lim\limits_{\lambda \to \infty} e^{-y v_t(\lambda)}= \lim
\limits_{\lambda \to \infty} \mathbb{E}\left[e^{-\lambda Z_t^{y}} \right] \notag \\
&= \lim\limits_{\lambda \to \infty} \left( \mathbb{E}\left[ e^{
-\lambda Z_t^{y}} \mathbbm{1}_{\left\lbrace Z_t^{y} = 0\right
\rbrace}\right] + \mathbb{E} \left[ e^{-\lambda Z_t^{y}}
\mathbbm{1}_{\left\lbrace Z_t^{y} > 0\right\rbrace} \right] \right) \notag \\
&= \mathbb{P}\left(Z_t^{y} = 0\right) >0 \label{meaning:d}
\end{align}
for all $t > 0$ and $y \geqslant 0$.

\subsection{Special case ii): $y = 0$.} Consider $(Y_t^0)_{t\geqslant0}$ that satisfies
\begin{equation}\label{eq:y_t^0}
\mathrm{d}Y_t^{0} = (a-bY_t^{0}) \mathrm{d}t + \sqrt[\alpha]{
Y_{t-}^{0}}\mathrm{d}L_t, \quad t \geqslant 0, \quad Y^0_0=0.
\end{equation}
In view of \eqref{eq:charfunc}, we easily see that the Laplace transform of $Y_t^0$
equals $\varphi_2(t,\lambda)$.

\section{Transition densities of the $\alpha$-root process $Y$}

In this section we show that the $\alpha$-root process $Y$ has positive
and continuous transition densities. Our approach is essentially based
on the inverse Fourier transform.

Recall that the function $v_{t}(\cdot)$ given by (\ref{defi:v_t})
is defined on $\mathbb{R}_{>0}$. By considering the complex power functions,
the domain of definition for $v_{t}(\cdot)$ can be extended to $\mathbb{C}\setminus\left\{ 0\right\} $.
Indeed, the function
\begin{equation}
v_{t}(z)=\left(\left(\frac{1}{\alpha b}+z^{(1-\alpha)}\right)e^{b(\alpha-1)t}-\frac{1}{\alpha b}\right)^{\frac{1}{1-\alpha}},\quad z\in\mathbb{C}\setminus\left\{ 0\right\} ,\label{defi: complex v_s}
\end{equation}
is well-defined, where the complex power function is given by (\ref{defi: complex power}).

We next establish two estimates on $\int_{0}^{t}v_{s}(z)\mathrm{d}s$.
Since the proofs are of pure analytic nature, we put them in the appendix.
\begin{lem}
Let $T>1$. Then there exists a sufficiently small constant $\varepsilon_{0}>0$
such that
\begin{equation}
\mathrm{Re}\left(\int_{0}^{t}v_{s}(z)\mathrm{d}s\right)\geqslant-C_{1}+C_{2}\vert z\vert^{2-\alpha}\label{esti: aim lem1 in appendix}
\end{equation}
when $|\mathrm{Arg}(z)|\in[\pi/2-\varepsilon_{0},\pi/2+\varepsilon_{0}]$
and $T^{-1}\leqslant t\leqslant T$, where $C_{1},\,C_{2}>0$ are
constants depending only on $a,\,b,\,\alpha,\,\varepsilon_{0}$ and
$T$.\label{lem1: pure esti}\end{lem}
\begin{proof}
See the appendix. \end{proof}
\begin{lem}
Let $\varepsilon_{0}$ be as in the previous lemma. Then for each
$t\geqslant0$, we can find constants $C_{3},\,C_{4}>0$, which depend
only on $a,\,b,\,\alpha,\,\varepsilon_{0}$ and $t$, such that
\[
\left\vert \int_{0}^{t}v_{s}(z)\mathrm{d}s\right\vert \leqslant C_{3}+C_{4}\vert z\vert^{2-\alpha}
\]
when $\mathrm{Arg}(z)\in[\pi/2+\varepsilon_{0},\pi]$ and $|z|\geqslant2$.
\label{lem2: pure esti}\end{lem}
\begin{proof}
See the appendix.
\end{proof}
Now, consider the process $(Y_{t}^{0})_{t\geqslant0}$ given by \eqref{eq:y_t^0}.
As shown in \cite[p.\,257]{MR1145236}, the function
\[
\mathbb{E}\left[\exp\left(-uY_{t}^{0}\right)\right],\quad u\in\mathcal{U}_{+},
\]
is continuous on $\mathcal{U}_{+}$ and holomorphic on $\mathcal{U}_{+}^{\mathrm{o}}$.
On the other hand, the function $z\mapsto v_{t}(z)$ given in (\ref{defi: complex v_s})
is continuous on $\mathcal{U}_{+}$ and holomorphic on $\mathcal{U}_{+}^{\mathrm{o}}$
for each $t\geqslant0$. Therefore, we have
\begin{equation}
\mathbb{E}\left[e^{-uY_{t}^{0}}\right]=\exp\left\lbrace -a\int_{0}^{t}v_{s}(u)\mathrm{d}s\right\rbrace ,\quad u\in\mathcal{U}_{+}.\label{eq: gene Lap tran for Y^0_t}
\end{equation}
Indeed, the equality (\ref{eq: gene Lap tran for Y^0_t}) is true
at least for $u\in\mathbb{R}_{>0}$ by (\ref{eq:charfunc}). This
and the identity theorem for holomorphic functions (see e.g. \cite[Theorem III.3.2]{MR2513384}) imply (\ref{eq: gene Lap tran for Y^0_t})
for all $u\in\mathcal{U}_{+}$, since both sides of (\ref{eq: gene Lap tran for Y^0_t})
are functions that are continuous on $\mathcal{U}_{+}$ and holomorphic
on $\mathcal{U}_{+}^{\mathrm{o}}$. In particular, the characteristic
function of $Y_{t}^{0}$ with $t>0$ is given by
\[
\mathbb{E}\left[e^{i\xi Y_{t}^{0}}\right]=\exp\left\lbrace -a\int_{0}^{t}v_{s}(i\xi)\mathrm{d}s\right\rbrace ,\quad\xi\in\mathbb{R}.
\]

In the next lemma we obtain the existence of a density function for
$Y_{t}^{0}$ when $t>0$. Note that by \cite[Theorem 1.1]{MR3254346},
we have $Y_{t}^{0}\geqslant0$ a.s. for each $t\geqslant0$.

\begin{lem}\label{lem:4.1} Assume $a>0$ and $b>0$. Then for each
$t>0$, $Y_{t}^{0}$ possesses a density function $f_{Y_{t}^{0}}$
given by
\begin{equation}
f_{Y_{t}^{0}}(x):=\frac{1}{2\pi}\int_{-\infty}^{\infty}e^{-ix\xi}\exp\left\lbrace -a\int_{0}^{t}v_{s}(-i\xi)\mathrm{d}s\right\rbrace \mathrm{d}\xi,\quad x\geqslant0.\label{eq:density of Y^0_t}
\end{equation}
Moreover, the function $f_{Y_{t}^{0}}(x)$ is jointly continuous in
$(t,x)\in(0,\infty)\times\mathbb{{R}}_{\geqslant0}$, and $f_{Y_{t}^{0}}(\cdot)\in C^{\infty}(\mathbb{R}_{\geqslant0})$
for each $t>0$. \end{lem}

\begin{proof} Let $T>1$ be fixed. By Lemma \ref{lem1: pure esti},
there exist constants $c_{1},\,c_{2}>0$ such that
\begin{equation}
\left\vert \exp\left\lbrace -a\int_{0}^{t}v_{s}(-i\xi)\mathrm{d}s\right\rbrace \right\vert =\exp\left\lbrace \mathrm{Re}\left(-a\int_{0}^{t}v_{s}(-i\xi)\mathrm{d}s\right)\right\rbrace \leqslant c_{1}e^{-c_{2}\vert\xi\vert^{2-\alpha}}\label{eq:int_re1}
\end{equation}
for all $\xi\in\mathbb{R}$ and $t\in[1/T,T]$, which implies that
$\xi\mapsto\exp\lbrace-a\int_{0}^{t}v_{s}(-i\xi)\mathrm{d}s\rbrace$
is integrable on $\mathbb{R}$. Therefore, by the inversion formula
of Fourier transform, $Y_{t}^{0}$ has a density $f_{Y_{t}^{0}}$
given by (\ref{eq:density of Y^0_t}). The joint continuity of the
density $f_{Y_{t}^{0}}(x)$ in $(t,x)$ follows from (\ref{eq:int_re1}),
(\ref{eq:density of Y^0_t}) and dominated convergence theorem. The
smoothness property of $f_{Y_{t}^{0}}(\cdot)$ is a consequence of
(\ref{eq:int_re1}) and \cite[Proposition~28.1]{MR3185174}. \end{proof}

We remark that for each $t>0$, the function $f_{Y_{t}^{0}}(x)$ given
in (\ref{eq:density of Y^0_t}) is actually well-defined also for
$x<0$, although $f_{Y_{t}^{0}}(x)\equiv0$ for $x\leqslant0$, which
is due to the fact that $Y_{t}^{0}\geqslant0$ a.s.. Next, we would
like to know if $f_{Y_{t}^{0}}(x)>0$ when $x>0$. The next lemma
partly answers this question.

\begin{lem}\label{lem:4.2} For each $t>0$, the density function
$f_{Y_{t}^{0}}(\cdot)$ of $Y_{t}^{0}$ is almost everywhere positive
on $\mathbb{R}_{\geqslant0}$. \end{lem}

\begin{proof} Basically, the idea of the proof is as follows. We
will show the following:
\begin{claim*}
The function
\[
x\mapsto f_{Y_{t}^{0}}(x),\quad x\in\mathbb{R}_{>0},
\]
can be extended to a holomorphic function on $\mathcal{U}_{+}^{\mathrm{o}}$.
\end{claim*}
If this claim is true, then the set $A_{n}:=\lbrace x>1/n\thinspace:\thinspace f_{Y_{t}^{0}}(x)=0\rbrace$
with $n\in\mathbb{N}$ must be discrete, that is, for each $x\in A_{n}$,
one can find a neighbourhood of $x$ whose intersection with $A_{n}$
equals $x$; otherwise the identity theorem for holomorphic functions
implies that $f_{Y_{t}^{0}}(x)\equiv0$ for $x>0$. As a consequence,
$A_{n}$ is countable, which implies that $A:=\cup_{n\in\mathbb{N}}A_{n}$
is also countable and thus has Lebesgue measure $0$.

Let $x>0$ be fixed. We will complete the proof of the above claim
in several steps.\\

``\emph{Step 1}'': We derive a simpler representation for $f_{Y_{t}^{0}}(x)$.
We have
\begin{align}
f_{Y_{t}^{0}}(x) & =\frac{1}{2\pi}\int_{-\infty}^{\infty}e^{-ix\xi}\exp\left\lbrace -a\int_{0}^{t}v_{s}(-i\xi)\mathrm{d}s\right\rbrace \mathrm{d}\xi\nonumber \\
 & =\frac{1}{2\pi}\int_{0}^{\infty}e^{-ix\xi}\exp\left\lbrace -a\int_{0}^{t}v_{s}(-i\xi)\mathrm{d}s\right\rbrace \mathrm{d}\xi\nonumber \\
 & \qquad+\frac{1}{2\pi}\int_{-\infty}^{0}e^{-ix\xi}\exp\left\lbrace -a\int_{0}^{t}v_{s}(-i\xi)\mathrm{d}s\right\rbrace \mathrm{d}\xi\nonumber \\
 & =\frac{1}{2\pi}\int_{-\infty}^{0}e^{ix\xi}\exp\left\lbrace -a\int_{0}^{t}v_{s}(i\xi)\mathrm{d}s\right\rbrace \mathrm{d}\xi\nonumber \\
 & \qquad+\frac{1}{2\pi}\int_{-\infty}^{0}e^{-ix\xi}\exp\left\lbrace -a\int_{0}^{t}v_{s}(-i\xi)\mathrm{d}s\right\rbrace \mathrm{d}\xi.\label{eq1: congu-relation}
\end{align}
For $\xi<0$, we have
\begin{align*}
\overline{v_{s}(-i\xi)} & =\left(\left(\tfrac{1}{\alpha b}+\overline{(-i\xi)^{1-\alpha}}\right)e^{b(\alpha-1)s}-\tfrac{1}{\alpha b}\right)^{\tfrac{1}{1-\alpha}}\\
 & =\left(\left(\tfrac{1}{\alpha b}+(i\xi)^{1-\alpha}\right)e^{b(\alpha-1)s}-\tfrac{1}{\alpha b}\right)^{\tfrac{1}{1-\alpha}}=v_{s}(i\xi),
\end{align*}
which implies
\begin{equation}
\overline{e^{-ix\xi}\exp\left\lbrace -a\int_{0}^{t}v_{s}(-i\xi)\mathrm{d}s\right\rbrace }=e^{ix\xi}\exp\left\lbrace -a\int_{0}^{t}v_{s}(i\xi)\mathrm{d}s\right\rbrace .\label{eq2: congu-relation}
\end{equation}
By (\ref{eq1: congu-relation}) and (\ref{eq2: congu-relation}),
we get
\begin{equation}
f_{Y_{t}^{0}}(x)=\mathrm{Re}\left(\frac{1}{\pi}\int_{-\infty}^{0}e^{-ix\xi}\exp\left\lbrace -a\int_{0}^{t}v_{s}(-i\xi)\mathrm{d}s\right\rbrace \mathrm{d}\xi\right).\label{eq1:final f Y_t^0}
\end{equation}
For simplicity, let
\begin{equation}
I:=\frac{1}{\pi}\int_{-\infty}^{0}e^{-ix\xi}\exp\left\lbrace -a\int_{0}^{t}v_{s}(-i\xi)\mathrm{d}s\right\rbrace \mathrm{d}\xi.\label{eq2:final f Y_t^0}
\end{equation}
\\

``\emph{Step 2}'': We calculate $I$ by contour integration. By
a change of variables $z:=-i\xi$, we get
\begin{align}
I & =\frac{-i}{\pi}\int_{0}^{i\infty}e^{xz}\exp\left\lbrace -a\int_{0}^{t}v_{s}(z)\mathrm{d}s\right\rbrace \mathrm{d}z\nonumber \\
 & =\lim_{K\to\infty}\frac{-i}{\pi}\int_{iK^{-1}}^{iK}e^{xz}\exp\left\lbrace -a\int_{0}^{t}v_{s}(z)\mathrm{d}s\right\rbrace \mathrm{d}z.\label{eq1: convergence I}
\end{align}
Define two paths $\Gamma_{1,K}$ and $\Gamma_{2,K}$ by
\[
\Gamma_{1,K}(\vartheta):=Ke^{i\vartheta},\quad\vartheta\in\left[\frac{\pi}{2},\pi\right]\quad\text{and}\quad\Gamma_{2,K}(\vartheta):=K^{-1}e^{i\vartheta},\quad\vartheta\in\left[\frac{\pi}{2},\pi\right].
\]
According to (\ref{defi: complex v_s}), we see that the function
\[
z\mapsto e^{yz}\exp\left\lbrace -a\int_{0}^{t}v_{s}(z)\mathrm{d}s\right\rbrace ,\quad z\in\mathcal{O}_{1}:=\left\{ \rho\exp\left(i\vartheta\right)\,:\,\rho>0,\vartheta\in\left[\frac{\pi}{2},\pi\right]\right\} ,
\]
can be extended to a holomorphic function on $\mathcal{O}_{2}:=\{ \rho\exp\left(i\vartheta\right)\,:\,\rho>0,\vartheta\in(0,3\pi/2)\}$.
Therefore, we have
\begin{align}\label{eq:contour_int_2}
&\int_{iK^{-1}}^{iK}e^{xz}\exp\left\lbrace -a\int_{0}^{t}v_{s}(z)\mathrm{d}s\right\rbrace \mathrm{d}z\\
 & \quad\quad=\int_{-K^{-1}}^{-K}e^{xz}\exp\left\lbrace -a\int_{0}^{t}v_{s}(z)\mathrm{d}s\right\rbrace \mathrm{d}z-\int_{\Gamma_{1,K}}e^{xz}\exp\left\lbrace -a\int_{0}^{t}v_{s}(z)\mathrm{d}s\right\rbrace \mathrm{d}z\nonumber\\
 & \quad\quad\quad+\int_{\Gamma_{2,K}}e^{xz}\exp\left\lbrace -a\int_{0}^{t}v_{s}(z)\mathrm{d}s\right\rbrace \mathrm{d}z.\nonumber
\end{align}
 Since $\lim_{z\to0}e^{xz}\exp\left\lbrace -a\int_{0}^{t}v_{s}(z)\mathrm{d}s\right\rbrace =1$,
it follows that
\begin{equation}
\lim_{K\to\infty}\int_{\Gamma_{2,K}}e^{xz}\exp\left\lbrace -a\int_{0}^{t}v_{s}(z)\mathrm{d}s\right\rbrace \mathrm{d}z\bigg)=0.\label{eq3: convergence I}
\end{equation}

To estimate the second term on the right-hand side of (\ref{eq:contour_int_2}),
we divide the path $\Gamma_{1,K}$ into two parts, namely
\[
\Gamma_{11,K}(\vartheta):=Ke^{i\vartheta},\quad\vartheta\in\left[\frac{\pi}{2},\frac{\pi}{2}+\varepsilon_{0}\right]\quad\text{and}\quad\Gamma_{12,K}(\vartheta)=:Ke^{i\vartheta},\quad\vartheta\in\left[\frac{\pi}{2}+\varepsilon_{0},\pi\right],
\]
with $\varepsilon_{0}>0$ being the constant appearing in Lemmas \ref{lem1: pure esti}
and \ref{lem2: pure esti}. Then
\begin{align*}
&\int_{\Gamma_{1,K}}e^{xz}\exp\left\lbrace -a\int_{0}^{t}v_{s}(z)\mathrm{d}s\right\rbrace \mathrm{d}z\\
 & \quad=\int_{\Gamma_{11,K}}e^{xz}\exp\left\lbrace -a\int_{0}^{t}v_{s}(z)\mathrm{d}s\right\rbrace \mathrm{d}z+\int_{\Gamma_{12,K}}e^{xz}\exp\left\lbrace -a\int_{0}^{t}v_{s}(z)\mathrm{d}s\right\rbrace \mathrm{d}z\\
 & \quad:=II_{1}(K)+II_{2}(K).
\end{align*}
If we can show that $\lim_{K\to\infty}II_{1}(K)=0$ and $\lim_{K\to\infty}II_{2}(K)=0$,
then it follows from (\ref{eq1: convergence I}), (\ref{eq:contour_int_2})
and (\ref{eq3: convergence I}) that
\begin{equation}
I=\frac{-i}{\pi}\int_{0}^{-\infty}e^{xz}\exp\left\lbrace -a\int_{0}^{t}v_{s}(z)\mathrm{d}s\right\rbrace \mathrm{d}z.\label{eq3:final f Y_t^0}
\end{equation}
\\

``\emph{Step 3}'': We show that $\lim_{K\to\infty}II_{1}(K)=0$.
If $\vartheta\in\left[\pi/2,\pi/2+\varepsilon_{0}\right]$, then
\[
\left\vert e^{xKe^{i\vartheta}}\right\vert =e^{\mathrm{Re}\left(xKe^{i\vartheta}\right)}=e^{xK\cos(\vartheta)}\leqslant1.
\]
By Lemma \ref{lem1: pure esti}, we get
\begin{align}
\left\vert II_{1}(K)\right\vert  & =\left\vert \int_{\tfrac{\pi}{2}}^{\tfrac{\pi}{2}+\varepsilon_{0}}iKe^{i\vartheta}e^{xKe^{i\vartheta}}e^{-a\int_{0}^{t}v_{s}\left(Ke^{i\vartheta}\right)\mathrm{d}s}\mathrm{d}\vartheta\right\vert\nonumber \\
 & \leqslant K\int_{\tfrac{\pi}{2}}^{\tfrac{\pi}{2}+\varepsilon_{0}}\left\vert e^{-a\int_{0}^{t}v_{s}\left(Ke^{i\vartheta}\right)\mathrm{d}s}\right\vert \mathrm{d}\vartheta\leqslant K\varepsilon_{0}e^{aC_{1}-aC_{2}K^{2-\alpha}},
\label{eq:contour1zero_1}
\end{align}
which implies
\[
\lim_{K\to\infty}\left|II_{1}(K)\right|\leqslant\lim_{K\to\infty}K\varepsilon_{0}e^{aC_{1}-aC_{2}K^{2-\alpha}}=0.
\]
\\

``\emph{Step 4}'': We show that $\lim_{K\to\infty}II_{2}(K)=0$.
In case $\vartheta\in\left[\pi/2+\varepsilon_{0},\pi\right]$, then
\begin{equation}
\left\vert e^{xKe^{i\vartheta}}\right\vert =e^{\mathrm{Re}\left(xKe^{i\vartheta}\right)}=e^{xK\cos(\vartheta)}\leqslant e^{xK\cos\left(\tfrac{\pi}{2}+\varepsilon_{0}\right)}=e^{-xK\sin(\varepsilon_{0})}.\label{esti1: II(T)}
\end{equation}
So
\begin{align*}
 \left\vert II_{2}(K)\right\vert &=\left\vert \int_{\tfrac{\pi}{2}+\varepsilon_{0}}^{\pi}iKe^{i\vartheta}e^{xKe^{i\vartheta}}\exp\left\lbrace -a\int_{0}^{t}v_{s}\left(Ke^{i\vartheta}\right)\mathrm{d}s\right\rbrace \mathrm{d}\vartheta\right\vert \\
 &\leqslant K\int_{\tfrac{\pi}{2}+\varepsilon_{0}}^{\pi}\left\vert e^{xKe^{i\vartheta}}\right\vert \left\vert \exp\left\lbrace -a\int_{0}^{t}v_{s}\left(Ke^{i\vartheta}\right)\mathrm{d}s\right\rbrace \right\vert \mathrm{d}\vartheta\\
 & \leqslant Ke^{-xK\sin(\varepsilon_{0})}\int_{\tfrac{\pi}{2}+\varepsilon_{0}}^{\pi}\exp\left\lbrace a\left\vert \int_{0}^{t}v_{s}\left(Ke^{i\vartheta}\right)\mathrm{d}s\right\vert \right\rbrace \mathrm{d}\vartheta.
\end{align*}
By Lemma \ref{lem2: pure esti}, we get
\[
\lim_{K\to\infty}\left|II_{2}(K)\right|\leqslant\lim_{K\to\infty}K\left(\tfrac{\pi}{2}-\varepsilon_{0}\right)e^{-xK\sin(\varepsilon_{0})}e^{aC_{3}}e^{aC_{4}K^{2-\alpha}}=0.
\]
\\

``\emph{Step 5}'': By (\ref{eq1:final f Y_t^0}), (\ref{eq2:final f Y_t^0})
and (\ref{eq3:final f Y_t^0}), we get
\begin{align*}
f_{Y_{t}^{0}}(x) & =\mathrm{Re}\left(\frac{-i}{\pi}\int_{0}^{-\infty}e^{xz}\exp\left\lbrace -a\int_{0}^{t}v_{s}(z)\mathrm{d}s\right\rbrace \mathrm{d}z\right)\\
 & =\mathrm{Re}\left(\frac{i}{\pi}\int_{0}^{\infty}e^{-xz}\exp\left\lbrace -a\int_{0}^{t}v_{s}(-z)\mathrm{d}s\right\rbrace \mathrm{d}z\right)\\
 & =\mathrm{-Im}\left(\frac{1}{\pi}\int_{0}^{\infty}e^{-xz}\exp\left\lbrace -a\int_{0}^{t}v_{s}(-z)\mathrm{d}s\right\rbrace \mathrm{d}z\right)\\
 & =\frac{1}{\pi}\int_{0}^{\infty}e^{-xz}\left\{ -\mathrm{Im}\left(\exp\left\lbrace -a\int_{0}^{t}v_{s}(-z)\mathrm{d}s\right\rbrace \right)\right\} \mathrm{d}z.
\end{align*}

Let $x_{0}>0$ be fixed. By Lemma \ref{lem2: pure esti}, for $z\in\mathbb{R}_{\geqslant0}$
and $x\in\mathbb{C}$ with $\mathrm{Re}(x)\geqslant x_{0}$, we have
\begin{align}
\left|ze^{-xz}\mathrm{Im}\bigg(\exp\left\lbrace -a\int_{0}^{t}v_{s}(-z)\mathrm{d}s\right\rbrace \bigg)\right| & \leqslant ze^{-\mathrm{Re}\left(xz\right)}\left|\exp\left\lbrace -a\int_{0}^{t}v_{s}(-z)\mathrm{d}s\right\rbrace \right|\nonumber \\
 & \leqslant ze^{-x_{0}z}\left|\exp\left\lbrace -a\int_{0}^{t}v_{s}(z)\mathrm{d}s\right\rbrace \right|\notag\\
 & \leqslant ze^{-x_{0}z} \exp\left\lbrace aC_{3}+aC_{4}|z|^{2-\alpha}\right\rbrace ,\label{eq:int_real_integral}
\end{align}
where the right-hand side of (\ref{eq:int_real_integral}) is an integrable
function (with the variable $z$) on $\mathbb{R}_{\geqslant0}$. By
Lebesgue differential theorem, we see that the function
\[
x\mapsto\frac{1}{\pi}\int_{0}^{\infty}e^{-xz}\left\{ -\mathrm{Im}\left(\exp\left\lbrace -a\int_{0}^{t}v_{s}(-z)\mathrm{d}s\right\rbrace \right)\right\} \mathrm{d}z,\quad x\in\mathcal{U}_{+}^{\mathrm{o}},
\]
is holomorphic, which means that $x\mapsto f_{Y_{t}^{0}}(x)$ has
a holomorphic extension on $\mathcal{U}_{+}^{\mathrm{o}}$. This
completes the proof. \end{proof}

With the help of the previous lemma, we are now able to prove the
main result of this section. Recall that the process $(Y_{t}^{y})_{t\geqslant0}$ is given by (\ref{defi: Y^y_t}).

\begin{prop}\label{thm:lower_bound_y_t} Assume $a>0$ and $b>0$.
Then for each $y\geqslant0$ and $t>0$, $Y_{t}^{y}$ possesses a
density function $f_{Y_{t}^{y}}$ given by
\begin{equation}
f_{Y_{t}^{y}}(x):=\frac{1}{2\pi}\int_{-\infty}^{\infty}e^{-ix\xi}\exp\left\lbrace -a\int_{0}^{t}v_{s}(-i\xi)\mathrm{d}s-yv_{t}(-i\xi)\right\rbrace \mathrm{d}\xi,\quad x\geqslant0,\label{defi: density for Y^y_t}
\end{equation}
where $f_{Y_{t}^{y}}(\cdot)\in C^{\infty}(\mathbb{R}_{\geqslant0})$
and $f_{Y_{t}^{y}}(x)>0$ for all $x>0$. Moreover, the function $f_{Y_{t}^{y}}(x)$
is jointly continuous in $(t,y,x)\in(0,\infty)\times\mathbb{{R}}_{\geqslant0}\times\mathbb{{R}}_{\geqslant0}$.
\end{prop}

\begin{proof} In view of (\ref{eq:charfunc}) and (\ref{eq: conv for Y^y_t}),
we have
\begin{equation}
\mathbb{E}\left[e^{i\xi Y_{t}^{y}}\right] =\mathbb{E}\left[e^{i\xi Y_{t}^{0}}\right]\cdot\mathbb{E}\left[e^{i\xi Z_{t}^{y}}\right]
  =\exp\left\lbrace -a\int_{0}^{t}v_{s}(-i\xi)\mathrm{d}s-yv_{t}(-i\xi)\right\rbrace,
\label{eq2: conv for Y^y_t}
\end{equation}
where $\xi\in\mathbb{R}.$ It follows from (\ref{eq:int_re1}) that
\[
\left|\mathbb{E}\left[e^{i\xi Y_{t}^{y}}\right]\right|\leqslant\left|\mathbb{E}\left[e^{i\xi Y_{t}^{0}}\right]\right|\le c_{1}e^{-c_{2}\vert\xi\vert^{2-\alpha}}
\]
for all $\xi\in\mathbb{R}$ and $t\in[1/T,T]$, where $T>1$ and $c_{1},\,c_{2}>0$
are constants depending on $T$. It follows that for $t>0$, $Y_{t}^{y}$
has a density $f_{Y_{t}^{y}}$ given by (\ref{defi: density for Y^y_t}).
Proceeding in the same way as in Lemma \ref{lem:4.1}, we obtain the
desired continuity and smoothness properties of $f_{Y_{t}^{y}}$.

We next show that if $t>0$, then $f_{Y_{t}^{y}}(x)>0$ for all $x>0$.
According to (\ref{eq2: conv for Y^y_t}), we see that the law of
$Y_{t}^{y}$, denoted by $\mu_{Y_{t}^{y}}$, is the convolution of
the laws of $Z_{t}^{y}$ and $Y_{t}^{0}$, which we denote by $\mu_{Z_{t}^{y}}$
and $\mu_{Y_{t}^{0}}$, respectively. So $\mu_{Y_{t}^{y}}=\mu_{Z_{t}^{y}}\ast\mu_{Y_{t}^{0}}$.
From this we deduce that for all $x>0$,
\begin{align}
f_{Y_{t}^{y}}(x) & =\int_{\mathbb{R}_{\geqslant0}}f_{Y_{t}^{0}}(x-z)\mu_{Z_{t}^{y}}(\mathrm{d}z)\nonumber \\
 & =\int_{(0,\infty)}f_{Y_{t}^{0}}(x-z)\mu_{Z_{t}^{y}}(\mathrm{d}z)+f_{Y_{t}^{0}}(x)\mu_{Z_{t}^{y}}\left(\lbrace0\rbrace\right).\label{eq:lower_bound_density}
\end{align}
By Lemma \ref{lem:4.2}, the density function $f_{Y_{t}^{0}}(x)$
of $Y_{t}^{0}$ is strictly positive for almost all $x>0$. In the
following we consider a fixed $x>0$ and distinguish between two cases.\\

``\emph{Case 1}'': $f_{Y_{t}^{0}}(x)>0$. It follows from (\ref{eq:lower_bound_density})
that
\begin{equation}
f_{Y_{t}^{y}}(x)\geqslant f_{Y_{t}^{0}}(x)\mu_{Z_{t}^{y}}\left(\lbrace0\rbrace\right)>0,\label{eq:lower_bound_density_pos_q}
\end{equation}
where we used the fact that $\mu_{Z_{t}^{y}}\left(\lbrace0\rbrace\right)=\mathbb{P}(Z_{t}^{y}=0)>0$,
as shown in (\ref{meaning:d}).\\

``\emph{Case 2}'': $f_{Y_{t}^{0}}(x)=0$. Then $x\in A_{n}$ for
a large enough $n$, where the set $A_{n}$ is the same as in the
proof of Lemma \ref{lem:4.2}. Since $A_{n}$ is discrete, we can
find a small enough $\delta>0$ such that
\begin{equation}
f_{Y_{t}^{0}}(x-z)>0,\label{eq1:f Y_t^y>0}
\end{equation}
for all $z\in(0,\delta]$. We next show that $\mu_{Z_{t}^{y}}\left((0,\delta]\right)>0$. By
(\ref{defi:d}), (\ref{meaning:d}) and L'Hospital's Rule, we get
\begin{align}
\lim_{\lambda\to\infty}\bigg(\mathbb{E}\left[e^{-\lambda(Z_{t}^{y}-\delta)}\right] & -\mathbb{E}\left[e^{-\lambda(Z_{t}^{y}-\delta)}\mathbbm{1}_{\lbrace Z_{t}^{y}=0\rbrace}\right]\bigg)\nonumber \\
 & =\lim_{\lambda\to\infty}e^{\lambda\delta}\left(\mathbb{E}\left[e^{-\lambda Z_{t}^{y}}\right]-\mathbb{P}(Z_{t}^{y}=0)\right)\notag \\
 & =\lim_{\lambda\to\infty}e^{\lambda\delta}\left(e^{-yv_{t}(\lambda)}-e^{-yd}\right)\notag\\
 & =\lim_{\lambda\to\infty}\delta^{-1}e^{\lambda\delta}ye^{-yv_{t}(\lambda)}\left(v_{t}(\lambda)\right)^{\alpha}e^{b(\alpha-1)t}\lambda^{-\alpha}=\infty.\label{eq:limit_lambda_y_t}
\end{align}

Suppose that $\mathbb{P}(Z_{t}^{y}\in(0,\delta])=0$. Then we can
use dominated convergence theorem to get
\begin{align*}
\lim_{\lambda\to\infty}  & \bigg(\mathbb{E}\left[e^{-\lambda(Z_{t}^{y}-\delta)}\right]  -\mathbb{E}\left[e^{-\lambda(Z_{t}^{y}-\delta)}\mathbbm{1}_{\left\lbrace Z_{t}^{y}=0\right\rbrace }\right]\bigg)\\
 & =\lim_{\lambda\to\infty}\left(\mathbb{E}\left[e^{-\lambda(Z_{t}^{y}-\delta)}\mathbbm{1}_{\left\lbrace 0<Z_{t}^{y}\leqslant\delta\right\rbrace }\right]+\mathbb{E}\left[e^{-\lambda(Z_{t}^{y}-\delta)}\mathbbm{1}_{\left\lbrace Z_{t}^{y}>\delta\right\rbrace }\right]\right)=0,
\end{align*}
which contradicts (\ref{eq:limit_lambda_y_t}). Consequently, the
assumption that $\mathbb{P}(Z_{t}^{y}\in(0,\delta])=0$ is not true
and we thus get $\mathbb{P}\left(Z_{t}^{y}\in(0,\delta]\right)>0$.
Now, by (\ref{eq:lower_bound_density}) and (\ref{eq1:f Y_t^y>0}),
we get
\begin{equation}
f_{Y_{t}^{y}}(x)\geqslant\int_{(0,\delta]}f_{Y_{t}^{0}}(x-z)\mu_{Z_{t}^{y}}(\mathrm{d}z)>0.\label{eq:lower_bound_density_pos_q-1}
\end{equation}

Summarizing the above two cases, we have $f_{Y_{t}^{y}}(x)>0$ for
all $x>0$. This completes the proof. \end{proof}

\section{A Foster-Lyapunov function for $(Y,X)$}

We now turn back to the two-dimensional affine process $(Y,X)=(Y_{t},X_{t})_{t\geqslant0}$
defined in \eqref{eq:SDE}. Our aim of this section is to construct
a Foster-Lyapunov function for $(Y,X)$.\\

For a functional $\Phi(Y,X)$
based on the process $(Y,X)$, we use $\mathbb{E}_{(y,x)}[\Phi(Y,X)]$
to indicate that the process $(Y,X)$ considered under the expectation
is with the initial condition $(Y_{0},X_{0})=(y,x)$, where $(y,x)\in\mathbb{R}_{\geqslant0}\times\mathbb{R}$
is constant. The notation $\mathbb{P}_{(y,x)}(\Phi(Y,X)\in\cdot)$
is similarly defined.

\begin{lem}\label{lem:foster_lyapunov_fct} Let $h\in C^{\infty}(\mathbb{R},\mathbb{R})$
be such that $h(x)\geqslant1$ for all $x\in\mathbb{R}$ and $h(x)=\vert x\vert$
whenever $\vert x\vert\geqslant2$. Define
\[
V(y,x):=\beta y+h(x),\quad y\geqslant0,\,x\in\mathbb{R},
\]
where $\beta>0$ is a constant. If $\beta$ is sufficiently large,
then $V$ is a Foster-Lyapunov function for $(Y,X)$,
that is, there exist constants $c,M>0$ such that
\begin{equation}
\mathbb{E}_{(y,x)}[V(Y_{t},X_{t})]\leqslant e^{-ct}V(y,x)+\tfrac{M}{c}\label{eq:foster_lyapunov_2}
\end{equation}
for all $(y,x)\in\mathbb{R}_{\geqslant0}\times\mathbb{R}$ and $t\geqslant0$.
\end{lem}

\begin{proof} Define $g(t,y,x):=\exp\lbrace ct\rbrace V(y,x)$, where $c>0$ is a
constant to be determined later. It is easy to see that $g\in C^{2}(\mathbb{R}_{\geqslant0}\times\mathbb{R}_{\geqslant0}\times\mathbb{R},\mathbb{R})$.
We define the functions $g_{1}',\,g_{2}',\,g_{3}'$ and $g_{3,3}''$
by
\begin{align*}
g_{1}'(t,y,x):=\tfrac{\partial}{\partial t}g(t,y,x) & =ce^{ct}V(y,x),\quad g_{2}'(t,y,x):=\tfrac{\partial}{\partial y}g(t,y,x)=\beta e^{ct},\\
g_{3}'(t,y,x):=\tfrac{\partial}{\partial x}g(t,y,x) & =e^{ct}\tfrac{\partial}{\partial x}h(x),\quad g_{3,3}''(t,y,x):=\tfrac{\partial^{2}}{\partial x^{2}}g(t,y,x)=e^{ct}\tfrac{\partial^{2}}{\partial x^{2}}h(x).
\end{align*}

If the process $(Y_{t},X_{t})_{t\geqslant0}$ starts from $(y,x)$,
i.e., $(Y_{0},X_{0})=(y,x)$, then we can use the L\'evy-It\^o decomposition
of $(L_{t})_{t\geqslant0}$ in (\ref{eq: Levy ito decom for L_t})
to obtain that for each $t\geqslant0$,
\begin{equation}
\begin{cases}
Y_{t}=y+\int_{0}^{t}\gamma\sqrt[\alpha]{Y_{s}}\mathrm{d}s+\int_{0}^{t}(a-bY_{s})\mathrm{d}s\\
\quad\quad+\int_{0}^{t}\int_{\lbrace\vert z\vert<1\rbrace}z\sqrt[\alpha]{Y_{s-}}\tilde{N}(\mathrm{d}s,\mathrm{d}z)+\int_{0}^{t}\int_{\lbrace\vert z\vert\geqslant1\rbrace}z\sqrt[\alpha]{Y_{s-}}N(\mathrm{d}s,\mathrm{d}z),\\
X_{t}=x+\int_{0}^{t}(m-\theta X_{s})\mathrm{d}s+\int_{0}^{t}\sqrt{Y_{s}}\mathrm{d}B_{s},
\end{cases}\label{neweq1: Lemma 5.1}
\end{equation}
where $\gamma$, $N(\mathrm{d}s,\mathrm{d}z)$ and $\tilde{N}(\mathrm{d}s,\mathrm{d}z)$
are as in (\ref{eq: Levy ito decom for L_t}). By (\ref{neweq1: Lemma 5.1})
and applying It\^o's formula for $g$ (see \cite[Theorem 94]{S-R}),
we obtain that for each $t\geqslant0$,
\begin{align}
 & g(t,Y_{t},X_{t})-g(0,Y_{0},X_{0})\nonumber \\
 & \quad=\int_{0}^{t}g_{1}'(s,Y_{s},X_{s})\mathrm{d}s+\int_{0}^{t}g_{2}'(s,Y_{s},X_{s})\gamma\sqrt[\alpha]{Y_{s}}\mathrm{d}s\nonumber \\
 & \quad\quad+\int_{0}^{t}g_{2}'(s,Y_{s},X_{s})(a-bY_{s})\mathrm{d}s+\int_{0}^{t}g_{3}'(s,Y_{s},X_{s})(m-\theta X_{s})\mathrm{d}s\nonumber \\
 & \quad\quad+\frac{1}{2}\int_{0}^{t}g_{3,3}''(s,Y_{s},X_{s})Y_{s}\mathrm{d}s+\int_{0}^{t}g_{3}'(s,Y_{s},X_{s})\sqrt{Y_{s}}\mathrm{d}B_{s}\nonumber \\
 & \quad\quad+\int_{0}^{t}\int_{\lbrace\vert z\vert<1\rbrace}\left(g(s,Y_{s-}+z\sqrt[\alpha]{Y_{s-}},X_{s-})-g(s,Y_{s-},X_{s-})\right)\tilde{N}(\mathrm{d}s,\mathrm{d}z)\nonumber \\
 & \quad\quad+\int_{0}^{t}\int_{\lbrace\vert z\vert\geqslant1\rbrace}\left(g(s,Y_{s-}+z\sqrt[\alpha]{Y_{s-}},X_{s-})-g(s,Y_{s-},X_{s-})\right)N(\mathrm{d}s,\mathrm{d}z)\nonumber \\
 & \quad\quad+\int_{0}^{t}\int_{\lbrace\vert z\vert<1\rbrace}\bigg(g(s,Y_{s}+z\sqrt[\alpha]{Y_{s}},X_{s})\nonumber \\
 & \quad\quad\quad\quad\quad\quad\quad\quad\quad\quad-g(s,Y_{s},X_{s})-z\sqrt[\alpha]{Y_{s}}g_{2}'(s,Y_{s},X_{s})\bigg)C_{\alpha}z^{-1-\alpha}\mathrm{d}s\mathrm{d}z\nonumber \\
 & \quad=\int_{0}^{t}(\mathcal{L}g)(s,Y_{s},X_{s})\mathrm{d}s+\int_{0}^{t}g_{1}'(s,Y_{s},X_{s})\mathrm{d}s+M_{t}\left(g\right),\label{eq1: Ito formula for g}
\end{align}
where
\begin{align*}
M_{t}(g) & :=\int_{0}^{t}g_{3}'(s,Y_{s},X_{s})\sqrt{Y_{s}}\mathrm{d}B_{s}\\
 & \quad+\int_{0}^{t}\int_{\lbrace\vert z\vert<1\rbrace}\left(g(s,Y_{s-}+z\sqrt[\alpha]{Y_{s-}},X_{s-})-g(s,Y_{s-},X_{s-})\right)\tilde{N}(\mathrm{d}s,\mathrm{d}z)\\
 & \quad+\int_{0}^{t}\int_{\lbrace\vert z\vert\geqslant1\rbrace}\left(g(s,Y_{s-}+z\sqrt[\alpha]{Y_{s-}},X_{s-})-g(s,Y_{s-},X_{s-})\right)N(\mathrm{d}s,\mathrm{d}z)\\
 & \quad-\int_{0}^{t}\int_{\lbrace\vert z\vert\geqslant1\rbrace}\left(g(s,Y_{s}+z\sqrt[\alpha]{Y_{s}},X_{s})-g(s,Y_{s},X_{s})\right)\hat{N}(\mathrm{d}s,\mathrm{d}z)
\end{align*}
and $\mathcal{L}g$ is defined by
\begin{align*}
(\mathcal{L}g)(t,y,x) & :=(a-by)g_{2}'(t,y,x)+(m-\theta x)g_{3}'(t,y,x)+\frac{1}{2}yg_{3,3}''(t,y,x)\\
 & \quad+\int_{\lbrace\vert z\vert<1\rbrace}\left(g(t,y+z\sqrt[\alpha]{y},x)-g(t,y,x)-z\sqrt[\alpha]{y}g_{2}'(t,y,x)\right)C_{\alpha}z^{-1-\alpha}\mathrm{d}z\\
 & \quad+\int_{\lbrace\vert z\vert\geqslant1\rbrace}\left(g(t,y+z\sqrt[\alpha]{y},x)-g(t,y,x)\right)C_{\alpha}z^{-1-\alpha}\mathrm{d}z+\gamma\sqrt[\alpha]{y}g_{2}'(t,y,x)
\end{align*}
for $(t,y,x)\in\mathbb{R}_{\geqslant0}\times\mathbb{R}_{\geqslant0}\times\mathbb{R}$.
By a change of variable $\tilde{z}:=z\sqrt[\alpha]{y}$ and an easy
computation, we see that $\mathcal{L}g=\mathcal{A}g$, where $\mathcal{A}$
is given in \eqref{eq:generator_of_y_x}. As a result, it follows
from (\ref{eq1: Ito formula for g}) that for each $t\geqslant0$,
\begin{align}
 & g(t,Y_{t},X_{t})-g(0,Y_{0},X_{0})\nonumber \\
 & \quad=\int_{0}^{t}(\mathcal{A}g)(s,Y_{s},X_{s})\mathrm{d}s+\int_{0}^{t}g_{1}'(s,Y_{s},X_{s})\mathrm{d}s+M_{t}\left(g\right).\label{eq: g_t-g_0}
\end{align}

The rest of the proof is divided into three steps:\\

``\emph{Step 1}'': We show that $(M_{t}(g))_{t\geqslant0}$ is a
martingale with respect to the filtration $(\mathcal{F}_{t})_{t\geqslant0}$,
where $(\mathcal{F}_{t})_{t\geqslant0}$ is the same as in Sect. 2.
To achieve this, we can use the same argument as in \cite{MR3254346}.
Define
\begin{align*}
M_{t}^{1}(g) & :=\int_{0}^{t}g_{3}'(s,Y_{s},X_{s})\sqrt{Y_{s}}\mathrm{d}B_{s},\\
M_{t}^{2}(g) & :=\int_{0}^{t}\int_{\lbrace\vert z\vert<1\rbrace}\left(g(s,Y_{s-}+z\sqrt[\alpha]{Y_{s-}},X_{s-})-g(s,Y_{s-},X_{s-})\right)\tilde{N}(\mathrm{d}s,\mathrm{d}z),\\
 & \quad+\int_{0}^{t}\int_{\lbrace\vert z\vert\geqslant1\rbrace}\left(g(s,Y_{s-}+z\sqrt[\alpha]{Y_{s-}},X_{s-})-g(s,Y_{s-},X_{s-})\right)N(\mathrm{d}s,\mathrm{d}z)\\
 & \quad\thinspace-\int_{0}^{t}\int_{\lbrace\vert z\vert\geqslant1\rbrace}\left(g(s,Y_{s}+z\sqrt[\alpha]{Y_{s}},X_{s})-g(s,Y_{s},X_{s})\right)\hat{N}(\mathrm{d}s,\mathrm{d}z),
\end{align*}
where $t\geqslant0$. By noting that $g_{2}'$ and $g_{3}'$ are both
bounded, we can proceed in the same way as in \cite[Theorem 2.1]{MR3254346}
to prove that $(M_{t}^{1}(g))_{t\geqslant0}$,
\begin{align*}
M_{t}^{3,n}(g) & :=\int_{0}^{t}\int_{\lbrace\vert z\vert<1\rbrace}\bigg(g(s,Y_{s-}\wedge n+z\sqrt[\alpha]{Y_{s-}\wedge n},X_{s-})\\
 & \quad\quad\quad\quad\quad\quad\quad\quad\quad\quad\quad-g(s,Y_{s-}\wedge n,X_{s-})\bigg)\tilde{N}(\mathrm{d}s,\mathrm{d}z),\quad t\geqslant0, \text{ and}\\
M_{t}^{4,n}(g) & :=\int_{0}^{t}\int_{\lbrace\vert z\vert\geqslant1\rbrace}\bigg(g(s,Y_{s-}\wedge n+z\sqrt[\alpha]{Y_{s-}\wedge n},X_{s-})\\
 & \quad\quad\quad\quad\quad\quad\quad\quad\quad\quad\quad-g(s,Y_{s-}\wedge n,X_{s-})\bigg)N(\mathrm{d}s,\mathrm{d}z)\\
 & \quad\thinspace-\int_{0}^{t}\int_{\lbrace\vert z\vert\geqslant1\rbrace}\bigg((g(s,Y_{s}\wedge n+z\sqrt[\alpha]{Y_{s}\wedge n},X_{s})\\
 & \quad\quad\quad\quad\quad\quad\quad\quad\quad\quad\quad-g(s,Y_{s}\wedge n,X_{s})\bigg)\hat{N}(\mathrm{d}s,\mathrm{d}z),\quad t\geqslant0,
\end{align*}
are all martingales with respect to the filtration $(\mathcal{F}_{t})_{t\geqslant0}$,
where $n\in\mathbb{N}$ is arbitrary. We omit the details here. For
each $n\in\mathbb{N}$, define
\begin{equation}
\eta_{t}^{n}(g):=M_{t}^{2}(g)-M_{t}^{3,n}(g)-M_{t}^{4,n}(g),\quad t\geqslant0.\label{eq 0: martingale prop M^2}
\end{equation}
Noting that $g(s,y+z,x)-g(s,y,x)=\beta z\exp\left(ct\right)$, we
get
\begin{align*}
\eta_{t}^{n}(g) & =\int_{0}^{t}\int_{\lbrace\vert z\vert<1\rbrace}\mathbbm{1}_{\lbrace Y_{s-}>n\rbrace}e^{cs}\beta z\sqrt[\alpha]{Y_{s-}}\tilde{N}(\mathrm{d}s,\mathrm{d}z)\\
 & \quad+\int_{0}^{t}\int_{\lbrace\vert z\vert\geqslant1\rbrace}\mathbbm{1}_{\lbrace Y_{s-}>n\rbrace}e^{cs}\beta z\sqrt[\alpha]{Y_{s-}}N(\mathrm{d}s,\mathrm{d}z)\\
 & \quad-\int_{0}^{t}\int_{\lbrace\vert z\vert\geqslant1\rbrace}\mathbbm{1}_{\lbrace Y_{s-}>n\rbrace}e^{cs}\beta z\sqrt[\alpha]{Y_{s-}}\hat{N}(\mathrm{d}s,\mathrm{d}z)\\
 & =\int_{0}^{t}\mathbbm{1}_{\lbrace Y_{s-}>n\rbrace}e^{cs}\beta\sqrt[\alpha]{Y_{s-}}\mathrm{d}L_{s},\quad\quad t\geqslant0,
\end{align*}
where we used the L\'evy-It\^o decomposition in \eqref{eq: Levy ito decom for L_t}
to get the second equality. It follows from \cite[Remark A.8]{MR3343292}
that for each $t\geqslant0$, there exist some constant $c_{1}>0$
such that
\[
\mathbb{E}_{(y,x)}\left[\sup_{s\in[0,t]}\left\vert \eta_{s}^{n}(g)\right\vert \right]\leqslant c_{1}\mathbb{E}_{(y,x)}\left[\left(\int_{0}^{t}\mathbbm{1}_{\lbrace Y_{s}>n\rbrace}Y_{s}\mathrm{d}s\right)^{\tfrac{1}{\alpha}}\right].
\]
Since $\mathbb{E}_{(y,x)}[Y_{s}]\leqslant c_{2}\left(1+y\exp(-bs/\alpha)\right)$
for all $s\in[0,t]$ by \cite[Proposition 2.8]{MR3343292}, where
$c_{2}>0$ is some constant, it follows that $\mathbb{E}_{(y,x)}\left[\int_{0}^{t}Y_{s}\mathrm{d}s\right]<\infty$
and further $\mathbb{E}_{(y,x)}\left[\left(\int_{0}^{t}Y_{s}\mathrm{d}s\right)^{1/\alpha}\right]<\infty$.
Therefore, by the dominated convergence theorem, we obtain
\begin{equation}
\lim_{n\to\infty}\mathbb{E}_{(y,x)}\left[\sup_{s\in[0,t]}\left\vert \eta_{s}^{n}(g)\right\vert \right]\leqslant c_{1}\lim_{n\to\infty}\mathbb{E}_{(y,x)}\left[\left(\int_{0}^{t}\mathbbm{1}_{\lbrace Y_{s}>n\rbrace}Y_{s}\mathrm{d}s\right)^{\tfrac{1}{\alpha}}\right]=0.\label{eq 1: martingale prop M^2}
\end{equation}
As shown in the proof of \cite[Theorem 2.1]{MR3254346}, the martingale
property of $(M_{t}^{2}(g))_{t\geqslant0}$ now follows from (\ref{eq 0: martingale prop M^2}),
(\ref{eq 1: martingale prop M^2}) and the fact that both $(M_{t}^{3,n}(g))_{t\geqslant0}$
and $(M_{t}^{4,n}(g))_{t\geqslant0}$ are martingales. It is clear
that $(M_{t}(g))_{t\geqslant0}=(M_{t}^{1}(g)+M_{t}^{2}(g))_{t\geqslant0}$
is also a martingale with respect to the filtration $(\mathcal{F}_{t})_{t\geqslant0}$.\\

``\emph{Step 2}'': We determine the constant $c>0$ and find another constant $M>0$ such that
\begin{equation}
(\mathcal{A}V)(y,x)\leqslant-cV(y,x)+M\label{eq:foster_lyapunov_1}
\end{equation}
for all $(y,x)\in\mathbb{R}_{\geqslant0}\times\mathbb{R}$, where $\mathcal{A}$ is given by \eqref{eq:generator_of_y_x}. For
the function $V$, we have $V\in C^{2}(\mathbb{R}_{\geqslant0}\times\mathbb{R},\mathbb{R})$,
\[
\tfrac{\partial}{\partial y}V(y,x)=\beta,\quad\tfrac{\partial}{\partial x}V(y,x)=\tfrac{\partial}{\partial x}h(x)=\begin{cases}
\tfrac{x}{\vert x\vert}, & \text{if }\vert x\vert>2\\
h'(x), & \text{if }\vert x\vert\leqslant2,
\end{cases}
\]
and
\[
\tfrac{\partial^{2}}{\partial x^{2}}V(y,x)=\tfrac{\partial^{2}}{\partial x^{2}}h(x):=\begin{cases}
0, & \text{if }\vert x\vert>2,\\
h''(x), & \text{if }\vert x\vert\leqslant2,
\end{cases}
\]
where $h'$ and $h''$ denote the first and second order derivatives
of the function $h$, respectively. So
\begin{align*}
(\mathcal{A}V)(y,x) & =(a-by)\beta+(m-\theta x)\tfrac{\partial}{\partial x}h(x)+\tfrac{1}{2}y\tfrac{\partial^{2}}{\partial x^{2}}h(x)\\
 & \quad+y\int_{0}^{\infty}\left(\beta(y+z)+h(x)-\beta y-h(x)-z\beta\right)C_{\alpha}z^{-1-\alpha}\mathrm{d}z\\
 & =(a-by)\beta+(m-\theta x)\tfrac{\partial}{\partial x}h(x)+\tfrac{1}{2}y\tfrac{\partial^{2}}{\partial x^{2}}h(x).
\end{align*}
By choosing $\beta>0$ large enough, we obtain that for all $(y,x)\in\mathbb{R}_{\geqslant0}\times\mathbb{R}$,
\begin{align}
(\mathcal{A}V)(y,x) & =a\beta-\tfrac{by\beta}{2}-\theta x\tfrac{\partial}{\partial x}h(x)+\left(-\tfrac{b\beta}{2}+\tfrac{1}{2}\tfrac{\partial^{2}}{\partial x^{2}}h(x)\right)y+m\tfrac{\partial}{\partial x}h(x)\nonumber \\
 & \leqslant a\beta-\tfrac{by\beta}{2}-\theta\left(h(x)\mathbbm{1}_{\lbrace x>2\rbrace}+h(x)\mathbbm{1}_{\lbrace x<-2\rbrace}\right)+0+c_{3}\nonumber \\
 & \leqslant a\beta-\tfrac{by\beta}{2}-\theta\left(h(x)\mathbbm{1}_{\lbrace\vert x\vert>2\rbrace}+h(x)\mathbbm{1}_{\lbrace\vert x\vert\leqslant2\rbrace}\right)+c_{4}\nonumber \\
 & =a\beta-\tfrac{by\beta}{2}-\theta h(x)+c_{4}=-\tfrac{b\beta}{2}y-\theta h(x)+c_{5},
\label{eq:foster_lyapunov_eq}
\end{align}
where we used the boundedness of $\vert h'\vert$, $\vert h''\vert$
and $\vert h\vert\mathbbm{1}_{\lbrace\vert x\vert\leqslant2\rbrace}$
to get the first and second inequality. Here $c_{3},\ c_{4}$ and
$c_{5}$ are some positive constants. Now, we see that \eqref{eq:foster_lyapunov_1}
holds with $c:=\min(b/2,\theta)$ and $M:=c_{5}$.
\\

``\emph{Step }3'': We prove \eqref{eq:foster_lyapunov_2}. By (\ref{eq: g_t-g_0}),
(\ref{eq:foster_lyapunov_1}) and the martingale property of $\left(M_{t}(g)\right)_{t\geqslant0}$,
we obtain
\begin{align*}
e^{ct} & \mathbb{E}_{(y,x)}\left[V(Y_{t},X_{t})\right]-V(y,x)\\
 & \quad=\mathbb{E}_{(y,x)}\left[g(t,Y_{t},X_{t})\right]-\mathbb{E}_{(y,x)}\left[g(0,Y_{0},X_{0})\right]\\
 & \quad=\mathbb{E}_{(y,x)}\left[\int_{0}^{t}\left(e^{cs}(\mathcal{A}V)(Y_{s},X_{s})+ce^{cs}V(Y_{s},X_{s})\right)\mathrm{d}s\right]\\
 & \quad\leqslant\mathbb{E}_{(y,x)}\left[\int_{0}^{t}\left(e^{cs}(-cV(Y_{s},X_{s})+M)+ce^{cs}V(Y_{s},X_{s})\right)\mathrm{d}s\right]\\
 & \quad=\mathbb{E}_{(y,x)}\left[\int_{0}^{t}Me^{cs}\mathrm{d}s\right]\leqslant\frac{M}{c}e^{ct}
\end{align*}
for all $(y,x)\in\mathbb{R}_{\geqslant0}\times\mathbb{R}$ and $t\geqslant0$,
which implies \eqref{eq:foster_lyapunov_2}. This completes the proof.
\end{proof}

\begin{rem} To see the existence of a function $h\in C^{\infty}(\mathbb{R},\mathbb{R})$
that fulfills the conditions of Lemma \ref{lem:foster_lyapunov_fct},
we can proceed in the following way: let $\rho\in C^{\infty}(\mathbb{R},\mathbb{R})$
be such that $\rho(x)=1$ for $x\geqslant2$, $\rho(x)=0$ for $x\leqslant1$
and $0\leqslant\rho(x)\leqslant1$ for $1\leqslant x\leqslant2$.
Define $F:\mathbb{R}\to\mathbb{R}$ by $F(x):=\int_{0}^{x}\rho(r)\mathrm{d}r$,
$x\in\mathbb{R}$. Then
\[
F(x)=\begin{cases}
0, & x\leqslant1,\\
\in[0,1], & 1<x\leqslant2,\\
x-2+\int_{1}^{2}\rho(r)\mathrm{d}r, & x>2.
\end{cases}
\]
We now define $h:\mathbb{R}\to\mathbb{R}$ by $h(x):=F(|x|)+2-F(2)$,
$x\in\mathbb{R}$. Then $h$ satisfies the conditions required in
Lemma \ref{lem:foster_lyapunov_fct}. \end{rem}

\section{Exponential ergodicity of $(Y,X)$}

In this section we prove our main result, namely, the exponential
ergodicity of the affine two factor model $(Y,X)=(Y_{t},X_{t})_{t\geqslant0}$.\\

Let $\Vert\cdot\Vert_{TV}$ denote the total variation norm for signed
measures on $\mathbb{R}_{\geqslant0}\times\mathbb{R}$, namely,
\[
\Vert\mu\Vert_{TV}:=\sup\left\lbrace \vert\mu(A)\vert\right\rbrace ,
\]
where $\mu$ is a signed measure on $\mathbb{R}_{\geqslant0}\times\mathbb{R}$ and the above supremum is running for all Borel sets $A$ in $\mathbb{R}_{\geqslant0}\times\mathbb{R}$.

Let $\mathbf{P}^{t}(y,x,\cdot):=\mathbb{P}_{(y,x)}\left((Y_{t},X_{t})\in\cdot\right)$
denote the distribution of $(Y_{t},X_{t})_{t\geqslant0}$
with the initial condition $(Y_0,X_0)=(y_0,x_0)\in\mathbb{R}_{\geqslant0}\times\mathbb{R}$.

By \cite[Theorem 3.1]{MR3254346} and the argument in \cite[p.80]{MR2779872},
there exists a unique invariant probability measure $\pi$ for the
two dimensional process $(Y_{t},X_{t})_{t\geqslant0}$. Roughly speaking,
if for each $(y,x)\in\mathbb{R}_{\geqslant0}\times\mathbb{R}$, the
convergence of the distribution $\mathbf{P}^{t}(y,x,\cdot)$ to $\pi$
as $t\to\infty$ is exponentially fast with respect to the total variation
norm, then we say that the process $(Y_{t},X_{t})_{t\geqslant0}$
is exponentially ergodic.

The main result of this paper is the following:

\begin{thm}\label{thm:ergodicity_of_y_x} Consider the two-dimensional
affine process $(Y,X)=(Y_{t},X_{t})_{t\geqslant0}$ defined by \eqref{eq:SDE}
with parameters $\alpha\in(1,2)$, $a>0$, $b>0$, $m\in\mathbb{R}$
and $\theta>0$. Then $(Y_{t},X_{t})_{t\geqslant0}$ is exponentially
ergodic, that is, there exist constants $\delta\in(0,\infty)$ and $B\in(0,\infty)$
such that
\begin{equation}
\Vert\mathbf{P}^{t}(y,x,\cdot)-\pi\Vert_{TV}\leqslant B\left(V(y,x)+1\right)e^{-\delta t}\label{eq:ergod_prop}
\end{equation}
for all $t\geqslant0$ and $(y,x)\in\mathbb{R}_{\geqslant0}\times\mathbb{R}$.
\end{thm}

\begin{proof} We basically follow the proof of \cite[Theorem 6.3]{MR3437080}.
The essential idea is to use the so called Foster-Lyapunov criteria
developed in \cite{MR1234295} for the geometric ergodicity of Markov
chains.

We first consider the skeleton chain $(Y_{n},X_{n})_{n\in\mathbb{Z}_{\geqslant0}}$,
which is a Markov chain on the state space $\mathbb{R}_{\geqslant0}\times\mathbb{R}$
with transition kernel $\mathbf{P}^{n}(y,x,\cdot)$. It is easy to
see that the measure $\pi$ is also an invariant probability measure
for the chain $(Y_{n},X_{n})_{n\in\mathbb{Z}_{\geqslant0}}$.

Let the function $V$ be the same as in Lemma \ref{lem:foster_lyapunov_fct}
and the constant $\beta>0$ there be sufficiently large. The Markov
property together with Lemma \ref{lem:foster_lyapunov_fct} implies
that
\begin{align*}
\mathbb{E}\big[V(Y_{n+1},X_{n+1})\thinspace & \vert\thinspace(Y_{0},X_{0}),(Y_{1},X_{1}),\ldots,(Y_{n},X_{n})\big]\\
 & =\int_{\mathbb{R}_{\geqslant0}}\int_{\mathbb{R}}V(y,x)\mathbf{P}^{1}(Y_{n},X_{n},\mathrm{d}y\mathrm{d}x)\leqslant e^{-c}V(Y_{n},X_{n})+\frac{M}{c},
\end{align*}
where $c$ and $M$ are the positive constants in Lemma \ref{lem:foster_lyapunov_fct}.
If we set $V_{0}:=V$ and $V_{n}:=V(Y_{n},X_{n})$, $n\in\mathbb{N}$,
then
\[
\mathbb{E}[V_{1}]\leqslant e^{-c}V_{0}(Y_0,X_0)+\frac{M}{c}
\]
and, for all $n\in\mathbb{N}$,
\[
\mathbb{E}\left[V_{n+1}\thinspace\vert\thinspace(Y_{0},X_{0}),(Y_{1},X_{1}),\ldots(Y_{n},X_{n})\right]\leqslant e^{-c}V_{n}+\frac{M}{c}.
\]

In order to apply \cite[Theorem 6.3]{MR1174380} for the chain $(Y_{n},X_{n})_{n\in\mathbb{Z}_{\geqslant0}}$,
it remains to verify the following conditions:
\begin{enumerate}
\item[(a)] the Lebesgue measure $\lambda$ on $\mathbb{R}_{\geqslant0}\times\mathbb{R}$
is an irreducibility measure for the chain $(Y_{n},X_{n})_{n\in\mathbb{Z}_{\geqslant0}}$;
\item[(b)] the chain $(Y_{n},X_{n})_{n\in\mathbb{Z}_{\geqslant0}}$ is aperiodic
(the definition of aperiodicity can be found in \cite[p.114]{MR2509253});
\item[(c)] all compact sets of the state space $\mathbb{R}_{\geqslant0}\times\mathbb{R}$
are petite (see \cite[p.500]{MR1234294} for a definition).
\end{enumerate}
We now proceed to prove (a)-(c).\\

In order to prove (a), we will use
the same argument as in \cite[Theorem 4.1]{MR3254346}. It is enough
to check that for each $(y_{0},x_{0})\in\mathbb{R}_{\geqslant0}\times\mathbb{R}$,
the measure $\mathbf{P}^{1}(y_{0},x_{0},\cdot)$ is absolutely continuous
with respect to the Lebesgue measure with a density function $p_{1}(y,x|y_{0},x_{0})$
that is strictly positive for almost all $(y,x)\in\mathbb{R}_{\geqslant0}\times\mathbb{R}$.
Indeed, let $A$ be a Borel set of $\mathbb{R}_{\geqslant0}\times\mathbb{R}$
with $\lambda(A)>0$. Then
\[
\mathbb{P}_{(y_{0},x_{0})}\left(\tau_{A}<\infty\right) \geqslant\mathbf{P}^{1}\left(y_{0},x_{0},A\right)
 =\iint_{A}p_{0}(y,x|y_{0},x_{0})\mathrm{d}y\mathrm{d}x>0
\]
for all $(y_{0},x_{0})\in\mathbb{R}_{\geqslant0}\times\mathbb{R}$,
where the stopping time $\tau_{A}$ is defined by $\tau_{A}:=\inf\lbrace n\geqslant0\thinspace:\thinspace(Y_{n},X_{n})\in A\rbrace$.

Next, we prove the existence of the density $p_{1}(y,x|y_{0},x_{0})$
with the required property. Recall that
\[
Y_{1}=e^{-b}\left(y_{0}+a\int_{0}^{1}e^{bs}\mathrm{d}s+\int_{0}^{1}e^{bs}\sqrt[\alpha]{Y_{s-}}\mathrm{d}L_{s}\right),
\]
and
\[
X_{1}=e^{-\theta}\left(x_{0}+m\int_{0}^{1}e^{\theta s}\mathrm{d}s+\int_{0}^{1}e^{\theta s}\sqrt{Y_{s}}\mathrm{d}B_{s}\right),
\]
provided that $(Y_{0},X_{0})=(y_{0},x_{0})\in\mathbb{R}_{\geqslant0}\times\mathbb{R}$.
For $(\bar{y},\bar{x})\in\mathbb{R}_{\geqslant0}\times\mathbb{R}$,
we have
\begin{align}
\mathbb{P}_{(y_{0},x_{0})}\left(Y_{1}<\bar{y},X_{1}<\bar{x}\right)
 & =\mathbb{E}_{(y_{0},x_{0})}\left[\mathbb{P}_{(y_{0},x_{0})}\left(\left.Y_{1}<\bar{y},X_{1}<\bar{x}\thinspace\right\vert \thinspace Y_{1}\right)\right]\notag\\
 & =\mathbb{E}_{(y_{0},x_{0})}\left[\mathbb{E}_{(y_{0},x_{0})}\left[\left.\mathbbm{1}_{\left\lbrace Y_{1}<\bar{y}\right\rbrace }\mathbbm{1}_{\left\lbrace X_{1}<\bar{x}\right\rbrace }\thinspace\right\vert \thinspace Y_{1}\right]\right]\nonumber \\
 & =\mathbb{E}_{(y_{0},x_{0})}\left[\mathbbm{1}_{\left\lbrace Y_{1}<\bar{y}\right\rbrace }\mathbb{E}_{(y_{0},x_{0})}\left[\left.\mathbbm{1}_{\left\lbrace X_{1}<\bar{x}\right\rbrace }\thinspace\right\vert \thinspace Y_{1}\right]\right].\label{eq1: cond. dist.}
\end{align}
Note that $(Y_{t})_{t\geqslant0}$ and the Brownian motion $(B_{t})_{t\geqslant0}$
are independent, since $(L_{t})_{t\geqslant0}$ and $(B_{t})_{t\geqslant0}$
are independent and $(Y_{t})_{t\geqslant0}$ is a strong solution.
Therefore, the conditional distribution of $X_{1}$ given $(Y_{t})_{t\in[0,1]}$
is a normal distribution with mean $x_{0}\exp(-\theta)+m\left(1-\exp(-\theta)\right)/\theta$
and variance $\exp\left(-2\theta\right)\int_{0}^{1}Y_{s}\exp\left(2\theta s\right)\mathrm{d}s$.
Hence, we get that for $\bar{x}\in\mathbb{R}$,
\begin{align}
&\mathbb{E}_{(y_{0},x_{0})}\left[\left.\mathbbm{1}_{\left\lbrace X_{1}<\bar{x}\right\rbrace }\thinspace\right\vert \thinspace Y_{1}\right]\nonumber \\
 & =\mathbb{E}_{(y_{0},x_{0})}\left[\left.\mathbb{E}_{(y_{0},x_{0})}\left[\left.\mathbbm{1}_{\left\lbrace X_{1}<\bar{x}\right\rbrace }\thinspace\right\vert \thinspace(Y_{t})_{0\leqslant t\leqslant1}\right]\thinspace\right\vert \thinspace Y_{1}\right]\notag\\
 & =\mathbb{E}_{(y_{0},x_{0})}\left[\left.\int_{-\infty}^{\bar{x}}\varrho\left(r-e^{-\theta}x_{0}-\tfrac{m}{\theta}\left(1-e^{-\theta}\right);e^{-2\theta}\int_{0}^{1}e^{2\theta s}Y_{s}\mathrm{d}s\right)\mathrm{d}r\,\right\vert \thinspace Y_{1}\right],\label{eq0: cond. dist.}
\end{align}
where $\varrho(r;\sigma^{2})$ is the density of the normal distribution
with variance $\sigma^{2}>0$, i.e.,
\[
\varrho(r;\sigma^{2}):=\frac{1}{\sigma\sqrt{2\pi}}e^{-\tfrac{r^2}{2\sigma^{2}}},\quad r\in\mathbb{R}.
\]
Note that the assumption $a>0$ ensures that
\[
\mathbb{P}_{(y_{0},x_{0})}\left(\int_{0}^{1}e^{2\theta s}Y_{s}\mathrm{d}s>0\right)=1.
\]
By \cite[Theorem 6.3]{MR1876169} and considering the conditional distribution
of $\int_{0}^{1}e^{2\theta s}Y_{s}\mathrm{d}s$ given $Y_{1}$, we
can find a probability kernel $K_{(y_{0},x_{0})}(\cdot,\cdot)$ from
$\mathbb{R}_{\geqslant0}$ to $\mathbb{R}_{\geqslant0}$ such that
\[
\mathbb{P}_{(y_{0},x_{0})}\left(\left.\int_{0}^{1}e^{2\theta s}Y_{s}\mathrm{d}s\in\cdot\thinspace\right\vert \thinspace Y_{1}\right)=K_{(y_{0},x_{0})}(Y_{1},\cdot)
\]
and
\begin{equation}
K_{(y_{0},x_{0})}(z,\mathbb{R}_{>0})=1, \quad \text{for all }z>0.\label{eq2: cond. dist.}
\end{equation}
So
\begin{align}
\mathbb{E}_{(y_{0},x_{0})} & \left[\left.\int_{-\infty}^{\bar{x}}\varrho\left(r-e^{-\theta}x_{0}-\tfrac{m}{\theta}\left(1-e^{-\theta}\right);e^{-2\theta}\int_{0}^{1}e^{2\theta s}Y_{s}\mathrm{d}s\right)\mathrm{d}r\,\right\vert \thinspace Y_{1}\right]\nonumber \\
 & =\int_{0}^{\infty}\left(\int_{-\infty}^{\bar{x}}\varrho\left(r-e^{-\theta}x_{0}-\tfrac{m}{\theta}\left(1-e^{-\theta}\right);e^{-2\theta}w\right)\mathrm{d}r\right)K_{(y_{0},x_{0})}(Y_{1},\mathrm{d}w)\nonumber \\
 & =\int_{-\infty}^{\bar{x}}\left(\int_{0}^{\infty}\varrho\left(r-e^{-\theta}x_{0}-\tfrac{m}{\theta}\left(1-e^{-\theta}\right);e^{-2\theta}w\right)K_{(y_{0},x_{0})}(Y_{1},\mathrm{d}w)\right)\mathrm{d}r.\label{eq3: cond. dist.}
\end{align}
It follows from (\ref{eq1: cond. dist.}), (\ref{eq0: cond. dist.})
and (\ref{eq3: cond. dist.}) that for all $(\bar{y},\bar{x})\in\mathbb{R}_{\geqslant0}\times\mathbb{R}$,
\begin{align}
\mathbb{P}_{(y_{0},x_{0})}\left(Y_{1}<\bar{y},X_{1}<\bar{x}\right) &=\int_{0}^{\bar{y}}\int_{-\infty}^{\bar{x}}\bigg (\int_{0}^{\infty}\varrho\left(r-e^{-\theta}x_{0}-\tfrac{m}{\theta}\left(1-e^{-\theta}\right);e^{-2\theta}w\right) \nonumber \\
&\qquad\qquad\qquad\qquad\qquad\quad\cdot K_{(y_{0},x_{0})}(z,\mathrm{d}w) \bigg )\mathrm{\mathit{f_{Y_{\mathrm{1}}^{y_{\mathrm{0}}}}\mathrm{(}z\mathrm{)}}\mathrm{\mathit{\mathrm{\mathrm{d}\mathit{r}d}}}}z,\label{eq4: cond. dist.}
\end{align}
where $f_{Y_{1}^{y_{0}}}$ is given in \eqref{defi: density for Y^y_t}.
Define
\[
p_{1}(y,x|y_{0},x_{0}):=f_{Y_{1}^{y_{0}}}(y)\int_{0}^{\infty}\varrho\left(x-e^{-\theta}x_{0}-\tfrac{m}{\theta}\left(1-e^{-\theta}\right);e^{-2\theta}w\right)K_{(y_{0},x_{0})}(y,\mathrm{d}w).
\]
By (\ref{eq2: cond. dist.}) and the fact that $f_{Y_{1}^{y_{0}}}(y)$
is strictly positive for all $y>0$ (see Theorem \ref{thm:lower_bound_y_t}),
for each $(y_{0},x_{0})\in\mathbb{R}_{\geqslant0}\times\mathbb{R}$,
the density $p_{1}(y,x|y_{0},x_{0})$ is strictly positive for almost
all $(y,x)\in\mathbb{R}_{\geqslant0}\times\mathbb{R}$. Moreover,
by (\ref{eq4: cond. dist.}), we have
\[
\mathbb{P}_{(y_{0},x_{0})}\left(Y_{1}<\bar{y},X_{1}<\bar{x}\right)=\int_{0}^{\bar{y}}\int_{-\infty}^{\bar{x}}p_{1}(y,x|y_{0},x_{0})\mathrm{d}y\mathrm{d}x
\]
for all $(\bar{y},\bar{x})\in\mathbb{R}_{\geqslant0}\times\mathbb{R}$.
So $p_{1}(\cdot,\cdot|y_{0},x_{0})$ is the density function of $(Y_{t},X_{t})$ given that $(Y_{0},X_{0})=(y_{0},x_{0})$.\\

To prove (b), i.e., the aperiodicity of the skeleton chain $(Y_{n},X_{n})_{n\in\mathbb{Z}_{\geqslant0}}$,
we use a contradiction argument. Suppose that the period $l$ of the
chain $(Y_{n},X_{n})_{n\in\mathbb{Z}_{\geqslant0}}$ is greater than
$1$ (see \cite[p.114]{MR2509253} for a definition of the period of a Markov chain). Then we can find disjoint Borel sets $A_{1},A_{2},\cdots,A_{l}$
such that
\begin{align}
\lambda(A_{i})>0,\ i&=1,\cdots,l,\quad\cup_{i=1}^{l}A_{i}=\mathbb{R}_{\geqslant0}\times\mathbb{R},\label{eq1: contra aperio}\\
&\thickspace\thinspace\mathbf{P}^{1}(y_{0},x_{0},A_{i+1})=1\label{eq2: contra aperio}
\end{align}
for all $(y_{0},x_{0})\in A_{i},\ i=1,\cdots,l-1$, and
\[
\mathbf{P}^{1}(y_{0},x_{0},A_{1})=1
\]
for all $(y_{0},x_{0})\in A_{l}$. By (\ref{eq2: contra aperio}), we have
\[
\iint_{(A_{2})^{c}}p_{1}(y,x|y_{0},x_{0})\mathrm{d}y\mathrm{d}x=0,\quad(y_{0},x_{0})\in A_{1},
\]
and further
\[
\iint_{A_{1}}p_{1}(y,x|y_{0},x_{0})\mathrm{d}y\mathrm{d}x=0,\quad(y_{0},x_{0})\in A_{1}.
\]
However, since for each $(y_{0},x_{0})\in\mathbb{R}_{\geqslant0}\times\mathbb{R}$,
the density $p_{1}(y,x|y_{0},x_{0})$ is strictly positive for almost
all $(y,x)\in\mathbb{R}_{\geqslant0}\times\mathbb{R}$, we must have
$\lambda(A_{1})=0$, which contradicts (\ref{eq1: contra aperio}).
Therefore, the assumption that $l\geqslant2$ is not true. So we have
$l=1$.\\

In view of \cite[Theorem 3.4 (ii)]{MR1174380}, to prove (c), it is
enough to check the Feller property of the skeleton chain $(Y_{n},X_{n})_{n\in\mathbb{Z}_{\geqslant0}}$.
By \cite[Theorem 2.7]{MR1994043}, the two-dimensional process $(Y_{t},X_{t})_{t\geqslant0}$,
as an affine process, possesses the Feller property. So the skeleton
chain $(Y_{n},X_{n})_{n\in\mathbb{Z}_{\geqslant0}}$ has also the
Feller property.\\

Now, we can apply \cite[Theorem 6.3]{MR1174380} and thus find constants
$\delta\in(0,\infty)$, $B\in(0,\infty)$ such that
\begin{equation}
\Vert\mathbf{P}^{n}(y,x,\cdot)-\pi\Vert_{TV}\leqslant B\left(V(y,x)+1\right)e^{-\delta n}\label{eq:ergod_prop_1}
\end{equation}
for all $n\in\mathbb{Z}_{\geqslant0},\thinspace(y,x)\in\mathbb{R}_{\geqslant0}\times\mathbb{R}$. For the remainder of the proof, i.e., to extend the inequality (\ref{eq:ergod_prop_1})
to all $t\geqslant0$, we can interpolate in the same way as in \cite[p.536]{MR1234295},
and we omit the details. This completes the proof. \end{proof}

\section*{Appendix}

\hspace*{\parindent}$Proof \ of \ Lemma \ \emph{\ref{lem1: pure esti}}$. We will complete the proof in three steps.\\

``\emph{Step 1}'': Consider $\rho\geqslant2$ and $\vartheta\in\left[\pi/2-\varepsilon,\pi/2+\varepsilon\right]$,
where $\varepsilon>0$ is a small constant whose exact value will
be determined later. We introduce a change of variables
\[
z:=\left(\left(\frac{1}{\alpha b}+\left(\rho e^{i\vartheta}\right){}^{(1-\alpha)}\right)e^{b(\alpha-1)s}-\frac{1}{\alpha b}\right)^{\frac{1}{1-\alpha}}
\]
and define $\Gamma_{0}:[0,t]\to\mathbb{C}$ by
\[
\Gamma_{0}(s):=\left(\left(\frac{1}{\alpha b}+\left(\rho e^{i\vartheta}\right)^{(1-\alpha)}\right)e^{b(\alpha-1)s}-\frac{1}{\alpha b}\right)^{\frac{1}{1-\alpha}},\quad s\in[0,t].
\]
Then we get
\begin{align}
\int_{0}^{t}v_{s}\left(\rho e^{i\vartheta}\right)\mathrm{d}s & ={\int_{0}^{t}\left(\left(\frac{1}{\alpha b}+\left(\rho e^{i\vartheta}\right){}^{1-\alpha}\right)e^{b(\alpha-1)s}-\frac{1}{\alpha b}\right)^{\tfrac{1}{1-\alpha}}\mathrm{d}s}\notag\\
 & =-\frac{1}{b}\int_{\Gamma_{0}}\left(1+\frac{z^{\alpha-1}}{\alpha b}\right)^{-1}\mathrm{d}z.\label{eq:int_v_s_i_xi}
\end{align}
Next, we derive a lower bound for $\mathrm{Re}\big(\int_{0}^{t}v_s\left(\rho e^{i\vartheta}\right)\mathrm{d}s\big)$.

Let $\Gamma_{0}^{*}$ be the range of $\Gamma_{0}$. Since $\Gamma_{0}^{*}\subset\mathcal{O}$
and $z\mapsto\left(1+z^{\alpha-1}/(\alpha b)\right)^{-1}$ is analytic
in $\mathcal{O}$, we have
\begin{equation}
\int_{\Gamma_{0}}\left(1+\frac{z^{\alpha-1}}{\alpha b}\right)^{-1}\mathrm{d}z=\int_{\rho e^{i\vartheta}}^{\left(\left(\frac{1}{\alpha b}+\left(\rho e^{i\vartheta}\right){}^{(1-\alpha)}\right)e^{b(\alpha-1)t}-\frac{1}{\alpha b}\right)^{\frac{1}{1-\alpha}}}\left(1+\frac{z^{\alpha-1}}{\alpha b}\right)^{-1}\mathrm{d}z.\label{eq1:int_Gamma0}
\end{equation}
Here and after, the notation
\[
\int_{w_{1}}^{w_{2}}\left(1+\frac{z^{\alpha-1}}{\alpha b}\right)^{-1}\mathrm{d}z
\]
means the integral $\int_{\Gamma_{[w_{1},w_{2}]}}\left(1+z^{\alpha-1}/(\alpha b)\right)^{-1}\mathrm{d}z$,
where $\Gamma_{[w_{1},w_{2}]}$ is the directed segment joining $w_{1}$
and $w_{2}$ and is defined by
\[
\Gamma_{[w_{1},w_{2}]}:[0,1]\to\mathbb{C}\quad\text{with}\quad\Gamma_{[w_{1},w_{2}]}(r):=(1-r)w_{1}+rw_{2},\quad r\in[0,1].
\]
By \eqref{eq:int_v_s_i_xi}, (\ref{eq1:int_Gamma0}) and the holomorphicity
of $z\mapsto\left(1+z^{\alpha-1}/(\alpha b)\right)^{-1}$ on $\mathcal{O}$,
we obtain
\begin{align}
\int_{0}^{t}v_{s}\left(\rho e^{i\vartheta}\right)\mathrm{d}s & =\frac{1}{b}\int_{e^{i\vartheta}}^{\rho e^{i\vartheta}}\left(1+\frac{z^{\alpha-1}}{\alpha b}\right)^{-1}\mathrm{d}z\nonumber\\
 & \quad\quad+\frac{1}{b}\int_{\left(\left(\frac{1}{\alpha b}+\left(\rho e^{i\vartheta}\right){}^{(1-\alpha)}\right)e^{b(\alpha-1)t}-\frac{1}{\alpha b}\right)^{\frac{1}{1-\alpha}}}^{e^{i\vartheta}}\left(1+\frac{z^{\alpha-1}}{\alpha b}\right)^{-1}\mathrm{d}z.
\label{eq:int_v_s_gamma_2-1}
\end{align}
Since the second term on the right-hand of \eqref{eq:int_v_s_gamma_2-1}
is continuous in $(t,\rho,\vartheta)\in[1/T,T]\times[2,\infty)\times[\pi/2-\varepsilon,\pi/2+\varepsilon]$
and converges to
\[
\frac{1}{b}\int_{\left(\left(e^{b(\alpha-1)t}-1\right)\frac{1}{\alpha b}\right)^{\frac{1}{1-\alpha}}}^{e^{i\vartheta}}\left(1+\frac{z^{\alpha-1}}{\alpha b}\right)^{-1}\mathrm{d}z
\]
(uniformly in $(t,\vartheta)\in[1/T,T]\times[\pi/2-\varepsilon,\pi/2+\varepsilon]$)
as $\rho\to\infty$, it must be bounded, i.e., we have
\begin{equation}
\left\vert \frac{1}{b}\int_{\left(\left(e^{b(\alpha-1)t}-1\right)\frac{1}{\alpha b}\right)^{\frac{1}{1-\alpha}}}^{e^{i\vartheta}}\left(1+\frac{z^{\alpha-1}}{\alpha b}\right)^{-1}\mathrm{d}z\right\vert \leqslant c_{3}\label{eq:bound_of_I_2-1}
\end{equation}
for all $t\in[1/T,T]$, $\vartheta\in[\pi/2-\varepsilon,\pi/2+\varepsilon]$
and $\rho\geqslant2$, where $c_{3}=c_{3}(\varepsilon,T)>0$ is some
constant.

Now, define $\Gamma_{\vartheta}:[0,1]\to\mathbb{C}$ by
\[
\Gamma_{\vartheta}(r):=(1-r)e^{i\vartheta}+r\rho e^{i\vartheta},\quad r\in[0,1],
\]
and let $\Gamma_{\vartheta}^{*}$ be the range of $\Gamma_{\vartheta}$.
We can calculate the real part of the first integral appearing on
the right-hand side of (\ref{eq:int_v_s_gamma_2-1}) by
\begin{align}
\mathrm{Re} & \left(\int_{e^{i\vartheta}}^{\rho e^{i\vartheta}}\left(1+\frac{z^{\alpha-1}}{\alpha b}\right)^{-1}\mathrm{d}z\right)\nonumber \\
 & =\mathrm{Re}\left(\int_{\Gamma_{\vartheta}}\left(1+\frac{z^{\alpha-1}}{\alpha b}\right)^{-1}\mathrm{d}z\right)\nonumber \\
 & =\mathrm{Re}\left(\int_{0}^{1}\left(1+\frac{\left(\Gamma_{v}(r)\right)^{\alpha-1}}{\alpha b}\right)^{-1}\partial_{r}\Gamma_{\vartheta}(r)\mathrm{d}r\right)\nonumber \\
 & =\mathrm{Re}\left(\int_{0}^{1}\frac{(\rho-1)e^{i\vartheta}}{1+\left(\Gamma_{\vartheta}(r)\right)^{\alpha-1}(\alpha b)^{-1}}\mathrm{d}r\right)\nonumber \\
 & =\int_{0}^{1}\left\vert \frac{(\rho-1)e^{i\vartheta}}{1+\left(\Gamma_{\vartheta}(r)\right)^{\alpha-1}(\alpha b)^{-1}}\right\vert \cos\left(\mathrm{Arg}\left(\frac{(\rho-1)e^{i\vartheta}}{1+\left(\Gamma_{\vartheta}(r)\right)^{\alpha-1}(\alpha b)^{-1}}\right)\right)\mathrm{d}r.\label{eq:estiofnormcos-1}
\end{align}
For $r\in[0,1]$, we have
\begin{align}
\mathrm{Arg}\left(1+\left(\Gamma_{\vartheta}(0)\right)^{\alpha-1}(\alpha b)^{-1}\right) & \leqslant\mathrm{Arg}\left(1+\left(\Gamma_{\vartheta}(r)\right)^{\alpha-1}(\alpha b)^{-1}\right)\nonumber \\
 & \leqslant\mathrm{Arg}\left(1+\left(\Gamma_{\vartheta}(1)\right)^{\alpha-1}(\alpha b)^{-1}\right).\label{eq1:cal-Arg}
\end{align}
Define $\delta_{\vartheta}$ by
\begin{align}
\delta_{\vartheta} & :=(\alpha-1)\vartheta-\mathrm{Arg}\left(1+\left(\Gamma_{\vartheta}(0)\right)^{\alpha-1}(\alpha b)^{-1}\right)\nonumber \\
 & \thinspace=(\alpha-1)\vartheta-\mathrm{Arg}\left(1+e^{i(\alpha-1)\vartheta}(\alpha b)^{-1}\right)\in(0,(\alpha-1)\vartheta).\label{eq2:cal-Arg}
\end{align}
It is easy to see that
\begin{equation}
\mathrm{Arg}\left(1+\left(\Gamma_{\vartheta}(1)\right)^{\alpha-1}(\alpha b)^{-1}\right)<(\alpha-1)\vartheta.\label{eq3:cal-Arg}
\end{equation}
By (\ref{eq1:cal-Arg}), (\ref{eq2:cal-Arg}) and (\ref{eq3:cal-Arg}),
we get
\[
\mathrm{Arg}\left(1+\left(\Gamma_{\vartheta}(r)\right)^{\alpha-1}(\alpha b)^{-1}\right)\in[(\alpha-1)\vartheta-\delta_{\vartheta},(\alpha-1)\vartheta),\quad r\in[0,1].
\]
As a result,
\begin{equation}
\mathrm{Arg}\left(\frac{(\rho-1)e^{i\vartheta}}{1+\left(\Gamma_{\vartheta}(r)\right)^{\alpha-1}(\alpha b)^{-1}}\right)\in\left((2-\alpha)\vartheta,(2-\alpha)\vartheta+\delta_{\vartheta}\right],\quad r\in[0,1].\label{Arg-range 1}
\end{equation}
Note that $0<\delta_{\pi/2}<(\alpha-1)\pi/2$ by (\ref{eq2:cal-Arg}).
Since $\delta_{\vartheta}$ is continuous in $\vartheta$, we see
that
\[
0<\lim_{\vartheta\to\frac{\pi}{2}}\left\{ (2-\alpha)\vartheta+\delta_{\vartheta}\right\} =(2-\alpha)\frac{\pi}{2}+\delta_{\frac{\pi}{2}}<\frac{\pi}{2}.
\]
Set
\[
c_{4}:=\frac{\pi}{2}-\left((2-\alpha)\frac{\pi}{2}+\delta_{\frac{\pi}{2}}\right)\in\left(0,\frac{\pi}{2}\right).
\]
Now, we choose the constant $\varepsilon_{0}>0$ small enough such
that
\begin{equation}
0<(2-\alpha)\vartheta<(2-\alpha)\vartheta+\delta_{\vartheta}\le\frac{\pi}{2}-\frac{c_{4}}{2}\label{Arg-range 2}
\end{equation}
for all $\vartheta\in\left[\pi/2-\varepsilon_{0},\pi/2+\varepsilon_{0}\right]$.
It follows from (\ref{Arg-range 1}) and (\ref{Arg-range 2}) that
for all $\vartheta\in\left[\pi/2-\varepsilon_{0},\pi/2+\varepsilon_{0}\right]$
and $r\in[0,1]$,
\begin{equation}
\cos\left(\mathrm{Arg}\left(\frac{(\rho-1)e^{i\vartheta}}{1+\left(\Gamma_{\vartheta}(r)\right)^{\alpha-1}(\alpha b)^{-1}}\right)\right)\geqslant\cos\left(\frac{\pi}{2}-\frac{c_{4}}{2}\right)=:c_{5}>0.\label{eq:estiofcostotoal-1}
\end{equation}
In view of (\ref{eq:estiofnormcos-1}) and (\ref{eq:estiofcostotoal-1}),
we get
\begin{align}
\mathrm{Re} & \left(\int_{e^{i\vartheta}}^{\rho e^{i\vartheta}}\left(1+\frac{z^{\alpha-1}}{\alpha b}\right)^{-1}\mathrm{d}z\right)\nonumber \\
 & \geqslant\cos\left(\frac{\pi}{2}-\frac{c_{4}}{2}\right)\int_{0}^{1}\left\vert \frac{(\rho-1)e^{i\vartheta}}{1+\left(\Gamma_{\vartheta}(r)\right)^{\alpha-1}(\alpha b)^{-1}}\right\vert \mathrm{d}r\nonumber \\
 & =c_{5}\int_{0}^{1}\frac{\rho-1}{\left\vert 1+\left(\Gamma_{\vartheta}(r)\right)^{\alpha-1}(\alpha b)^{-1}\right\vert }\mathrm{d}r\geqslant c_{5}\int_{0}^{1}\frac{\rho-1}{1+\left\vert \left(\Gamma_{\vartheta}(r)\right)^{\alpha-1}(\alpha b)^{-1}\right\vert }\mathrm{d}r\nonumber \\
 & =c_{5}\int_{0}^{1}\frac{\rho-1}{1+\left(1-r+r\rho\right)^{\alpha-1}(\alpha b)^{-1}}\mathrm{d}r=c_{5}\int_{0}^{\rho-1}\frac{1}{1+\left(1+r\right)^{\alpha-1}(\alpha b)^{-1}}\mathrm{d}r\nonumber \\
 & \geqslant\frac{c_{5}}{1+(\alpha b)^{-1}}\int_{0}^{\rho-1}\frac{1}{\left(1+r\right)^{\alpha-1}}\mathrm{d}r=c_{5}\alpha b(1+\alpha b)^{-1}(2-\alpha)^{-1}\left(\rho{}^{2-\alpha}-1\right).\label{eq:bound_of_I_1}
\end{align}
Combining (\ref{eq:int_v_s_gamma_2-1}), (\ref{eq:bound_of_I_2-1})
and \eqref{eq:bound_of_I_1} yields
\begin{equation}
\mathrm{Re}\left(\int_{0}^{t}v_{s}\left(\rho e^{i\vartheta}\right)\mathrm{d}s\right)\geqslant c_{6}\rho^{2-\alpha}-c_{7},\quad\rho\geqslant2,\,\vartheta\in\left[\frac{\pi}{2}-\varepsilon_{0},\frac{\pi}{2}+\varepsilon_{0}\right],\,t\in[1/T,T],\label{eq:estiofintvs1}
\end{equation}
where $c_{6},\,c_{7}>0$ are constants that depend only on $a,\,b,\,\alpha,\,\varepsilon_{0}$
and $T$. \\

``\emph{Step 2}'': The case with $\rho\geqslant2$ and $\vartheta\in\left[-\pi/2-\varepsilon_{0},-\pi/2+\varepsilon_{0}\right]$
can be similarly treated, and we thus get
\begin{equation}
\mathrm{Re}\left(\int_{0}^{t}v_{s}\left(\rho e^{i\vartheta}\right)\mathrm{d}s\right)\geqslant c_{8}\rho^{2-\alpha}-c_{9}\label{eq:estiofintvs2}
\end{equation}
for all $\rho\geqslant2,\,\vartheta\in\left[-\pi/2-\varepsilon_{0},-\pi/2+\varepsilon_{0}\right]$
and $t\in[1/T,T]$, where $c_{8},\ c_{9}>0$ are constants depending
only on $a,\,b,\,\alpha,\,\varepsilon_{0}$ and $T$.\\

``\emph{Step 3}'': Since $\int_{0}^{t}v_{s}\left(\rho e^{i\vartheta}\right)\mathrm{d}s$
is continuous in $(t,\rho,\vartheta)$, we can find a constant $c_{10}>0$
such that
\begin{equation}
\mathrm{Re}\left(\int_{0}^{t}v_{s}\left(\rho e^{i\vartheta}\right)\mathrm{d}s\right)\geqslant-c_{10}\label{eq:estiofintvs3}
\end{equation}
for all $0\leqslant\rho\leqslant2,\,\vartheta\in\left[-\pi/2-\varepsilon_{0},-\pi/2+\varepsilon_{0}\right]\cup\left[\pi/2-\varepsilon_{0},\pi/2+\varepsilon_{0}\right]$
and $t\in[1/T,T]$. The estimate (\ref{esti: aim lem1 in appendix})
now follows from \eqref{eq:estiofintvs1}, (\ref{eq:estiofintvs2})
and (\ref{eq:estiofintvs3}). \qed

\bigskip
\hspace*{\parindent}$Proof \ of \ Lemma \ \emph{\ref{lem2: pure esti}}$. Let $\rho\ge2$ and $\vartheta\in\left[\pi/2+\varepsilon_{0},\pi\right]$.
Our aim is to show
\begin{equation}
\left|\int_{0}^{t}v_{s}(\rho e^{i\vartheta})\mathrm{d}s\right|\leqslant C_{3}+C_{4}\rho^{2-\alpha}\label{aim lem2 pure esti}
\end{equation}
for some constants $C_{3},\,C_{4}>0$ that depend only on $a,\,b,\,\alpha,\,\varepsilon_{0}$
and $t$. Using the change of variables
\[
z:=\left(\frac{1}{\alpha b}+\left(\rho e^{i\vartheta}\right){}^{(1-\alpha)}\right)e^{b(\alpha-1)s}-\frac{1}{\alpha b},
\]
we get
\begin{align}
\int_{0}^{t}v_{s}(\rho e^{i\vartheta})\mathrm{d}s & =\int_{0}^{t}\left(\left(\frac{1}{\alpha b}+\left(\rho e^{i\vartheta}\right)^{1-\alpha}\right)e^{b(\alpha-1)s}-\frac{1}{\alpha b}\right)^{\tfrac{1}{1-\alpha}}\mathrm{d}s\nonumber \\
 & =\frac{1}{b(\alpha-1)}\int_{\left(\rho e^{i\vartheta}\right)^{1-\alpha}}^{\left(\frac{1}{\alpha b}+\left(\rho e^{i\vartheta}\right)^{1-\alpha}\right)e^{b(\alpha-1)t}-\frac{1}{\alpha b}}z^{\tfrac{1}{1-\alpha}}\left(z+\frac{1}{\alpha b}\right)^{-1}\mathrm{d}z.\label{eq:final lemma transform}
\end{align}
Since $\vartheta\in\left[\pi/2+\varepsilon_{0},\pi\right]$, we have
$(1-\alpha)\vartheta\in\left[(1-\alpha)\pi,(1-\alpha)(\pi/2+\varepsilon_{0})\right]$,
which implies
\begin{equation}
\left|\sin\left((1-\alpha)\vartheta\right)\right|\geqslant\min\left\{ \sin\left((\alpha-1)\pi\right),\sin\left((\alpha-1)(\pi/2+\varepsilon_{0})\right)\right\} =:c_{1}>0.\label{esti: sin (1-a)v}
\end{equation}
Note that $z\mapsto z^{1/(1-\alpha)}\left(z+1/(\alpha b)\right)^{-1}$
is holomorphic on $\mathcal{O}$. So we have
\begin{align}
 & \int_{(\rho e^{i\vartheta})^{1-\alpha}}^{(\tfrac{1}{\alpha b}+(\rho e^{i\vartheta})^{1-\alpha})e^{b(\alpha-1)t}-\tfrac{1}{\alpha b}}z^{\tfrac{1}{1-\alpha}}\left(z+\frac{1}{\alpha b}\right)^{-1}\mathrm{d}z\nonumber \\
 & \quad=\int_{(\rho e^{i\vartheta})^{1-\alpha}}^{(\rho e^{i\vartheta})^{1-\alpha}+2}z^{\tfrac{1}{1-\alpha}}\left(z+\frac{1}{\alpha b}\right)^{-1}\mathrm{d}z\nonumber \\
 & \qquad+\int_{(\rho e^{i\vartheta})^{1-\alpha}+2}^{\left(\tfrac{1}{\alpha b}+(\rho e^{i\vartheta})^{1-\alpha}\right)e^{b(\alpha-1)t}-\tfrac{1}{\alpha b}}z^{\tfrac{1}{1-\alpha}}\left(z+\frac{1}{\alpha b}\right)^{-1}\mathrm{d}z.\label{eq:int_bound}
\end{align}
Since
\begin{align*}
\lim_{\rho\to\infty}\int_{(\rho e^{i\vartheta})^{1-\alpha}+2}^{\left(\tfrac{1}{\alpha b}+(\rho e^{i\vartheta})^{1-\alpha}\right)e^{b(\alpha-1)t}-\tfrac{1}{\alpha b}}&z^{\tfrac{1}{1-\alpha}}\left(z+\frac{1}{\alpha b}\right)^{-1}\mathrm{d}z\\
&=\int_{2}^{\tfrac{1}{\alpha b}\left(e^{b(\alpha-1)t}-1\right)}z^{\tfrac{1}{1-\alpha}}\left(z+\frac{1}{\alpha b}\right)^{-1}\mathrm{d}z,
\end{align*}
where the convergence is uniform in $\vartheta\in\left[\pi/2+\varepsilon_{0},\pi\right]$,
we can find a constant $c_{2}>0$ such that
\begin{equation}
\left|\int_{(\rho e^{i\vartheta})^{1-\alpha}+2}^{\left(\tfrac{1}{\alpha b}+(\rho e^{i\vartheta})^{1-\alpha}\right)e^{b(\alpha-1)t}-\tfrac{1}{\alpha b}}z^{\tfrac{1}{1-\alpha}}\left(z+\frac{1}{\alpha b}\right)^{-1}\mathrm{d}z\right|\leqslant c_{2}\label{esti:final lemma}
\end{equation}
for all $\rho\geqslant2$ and $\vartheta\in\left[\pi/2+\varepsilon_{0},\pi\right]$.

We now proceed to estimate the first term on the right-hand side of
(\ref{eq:int_bound}). Define
\[
\Gamma_{\vartheta,\rho}(r):=\left(\rho e^{i\vartheta}\right)^{1-\alpha}+r,\quad r\in\left[0,2\right].
\]
By (\ref{esti: sin (1-a)v}), we have
\begin{equation}
|\rho^{1-\alpha}e^{(1-\alpha)i\vartheta}+r|\geqslant\rho^{1-\alpha}|\sin\left((1-\alpha)\vartheta\right)|\geqslant c_{1}\rho^{1-\alpha},\label{esti1: final}
\end{equation}
where $r\in[0,2]$ and $\vartheta\in\left[\pi/2+\varepsilon_{0},\pi\right]$.
If $r\in[2\rho^{1-\alpha},2],$ then
\begin{equation}
|\rho^{1-\alpha}e^{(1-\alpha)i\vartheta}+r|\geqslant r-\rho^{1-\alpha}\geqslant\frac{r}{2}.\label{esti2: final}
\end{equation}
It follows from (\ref{esti1: final}) and (\ref{esti2: final}) that
for $\rho\geqslant2$ and $\vartheta\in\left[\pi/2+\varepsilon_{0},\pi\right]$,
\begin{align}
 & \left|\int_{(\rho e^{i\vartheta})^{1-\alpha}}^{(\rho e^{i\vartheta})^{1-\alpha}+2}z^{\tfrac{1}{1-\alpha}}\left(z+\frac{1}{\alpha b}\right)^{-1}\mathrm{d}z\right|\nonumber \\
 & \quad=\left|\int_{0}^{2}\left(\Gamma_{\vartheta,\rho}(r)\right)^{\tfrac{1}{1-\alpha}}\left(\Gamma_{\vartheta,\rho}(r)+\frac{1}{\alpha b}\right)^{-1}\mathrm{d}r\right|\notag\\
 & \quad\leqslant c_{3}\int_{0}^{2}\vert\Gamma_{\vartheta,\rho}(r)\vert^{\tfrac{1}{1-\alpha}}\mathrm{d}r=c_{3}\int_{0}^{2}\left\vert \rho^{1-\alpha}e^{(1-\alpha)i\vartheta}+r\right\vert ^{\tfrac{1}{1-\alpha}}\mathrm{d}r\nonumber \\
 & \quad=c_{3}\int_{0}^{2\rho^{1-\alpha}}\left\vert \rho^{1-\alpha}e^{(1-\alpha)i\vartheta}+r\right\vert ^{\tfrac{1}{1-\alpha}}\mathrm{d}r\nonumber \\
 & \qquad+c_{3}\int_{2\rho^{1-\alpha}}^{2}\left\vert \rho^{1-\alpha}e^{(1-\alpha)i\vartheta}+r\right\vert ^{\tfrac{1}{1-\alpha}}\mathrm{d}r\nonumber \\
 & \quad\leqslant c_{3}\int_{0}^{2\rho^{1-\alpha}}\left(c_{1}\rho^{1-\alpha}\right)^{\tfrac{1}{1-\alpha}}\mathrm{d}r+c_{3}2^{1/(\alpha-1)}\int_{2\rho^{1-\alpha}}^{2}r^{\tfrac{1}{1-\alpha}}\mathrm{d}r\nonumber \\
 & \quad=2c_{3}c_{1}^{1/(1-\alpha)}\rho^{2-\alpha}+c_{3}2^{1/(\alpha-1)}\left.\tfrac{\alpha-1}{\alpha-2}r^{\tfrac{2-\alpha}{1-\alpha}}\right\vert _{r=2\rho^{1-\alpha}}^{2}\leqslant c_{4}\rho^{2-\alpha}+c_{5,}\label{eq:estofcontourint}
\end{align}
where $c_{3},\,c_{4},\,c_{5}>0$ are some constants. Combining (\ref{eq:final lemma transform}),
(\ref{eq:int_bound}), (\ref{esti:final lemma}) and (\ref{eq:estofcontourint})
yields (\ref{aim lem2 pure esti}). \qed \\

\textbf{Acknowledgements.} The author J. Kremer would like to thank the University of Wuppertal for the financial support through a doctoral funding program.

\bibliographystyle{amsplain}

\begin{thebibliography}{10}

\bibitem{MR2995525}
Mohamed~Ben Alaya and Ahmed Kebaier, \emph{Parameter estimation for the
  square-root diffusions: ergodic and nonergodic cases}, Stoch. Models
  \textbf{28} (2012), no.~4, 609--634. \MR{2995525}

\bibitem{MR3216637}
M{\'a}ty{\'a}s Barczy, Leif D{\"o}ring, Zenghu Li, and Gyula Pap,
  \emph{Parameter estimation for a subcritical affine two factor model}, J.
  Statist. Plann. Inference \textbf{151/152} (2014), 37--59. \MR{3216637}

\bibitem{MR3254346}
\bysame, \emph{Stationarity and ergodicity for an affine two-factor model},
  Adv. in Appl. Probab. \textbf{46} (2014), no.~3, 878--898. \MR{3254346}

\bibitem{MR3452993}
M{\'a}ty{\'a}s Barczy and Gyula Pap, \emph{Asymptotic properties of
  maximum-likelihood estimators for {H}eston models based on continuous time
  observations}, Statistics \textbf{50} (2016), no.~2, 389--417. \MR{3452993}

\bibitem{MR785475}
John~C. Cox, Jonathan~E. Ingersoll, Jr., and Stephen~A. Ross, \emph{A theory of
  the term structure of interest rates}, Econometrica \textbf{53} (1985),
  no.~2, 385--407.

\bibitem{MR1994043}
D.~Duffie, D.~Filipovi{\'c}, and W.~Schachermayer, \emph{Affine processes and
  applications in finance}, Ann. Appl. Probab. \textbf{13} (2003), no.~3,
  984--1053.

\bibitem{MR1793362}
Darrell Duffie, Jun Pan, and Kenneth Singleton, \emph{Transform analysis and
  asset pricing for affine jump-diffusions}, Econometrica \textbf{68} (2000),
  no.~6, 1343--1376. \MR{1793362 (2001m:91081)}

\bibitem{MR3264444}
Xan Duhalde, Cl{\'e}ment Foucart, and Chunhua Ma, \emph{On the hitting times of
  continuous-state branching processes with immigration}, Stochastic Process.
  Appl. \textbf{124} (2014), no.~12, 4182--4201. \MR{3264444}

\bibitem{MR1145236}
Gerald~B. Folland, \emph{Fourier analysis and its applications}, The Wadsworth
  \& Brooks/Cole Mathematics Series, Wadsworth \& Brooks/Cole Advanced Books \&
  Software, Pacific Grove, CA, 1992. \MR{1145236}

\bibitem{MR2513384}
Eberhard Freitag and Rolf Busam, \emph{Complex analysis}, second ed.,
  Universitext, Springer-Verlag, Berlin, 2009. \MR{2513384}

\bibitem{MR2584896}
Zongfei Fu and Zenghu Li, \emph{Stochastic equations of non-negative processes
  with jumps}, Stochastic Process. Appl. \textbf{120} (2010), no.~3, 306--330.
  \MR{2584896 (2011d:60178)}

\bibitem{Heston}
Steven~L. Heston, \emph{A closed-form solution for options with stochastic
  volatility with applications to bond and currency options}, Review of
  Financial Studies (1993), 6:327?343.

\bibitem{MR3167406}
Peng Jin, Vidyadhar Mandrekar, Barbara R{\"u}diger, and Chiraz Trabelsi,
  \emph{Positive {H}arris recurrence of the {CIR} process and its
  applications}, Commun. Stoch. Anal. \textbf{7} (2013), no.~3, 409--424.
  \MR{3167406}

\bibitem{MR3451177}
Peng Jin, Barbara R{\"u}diger, and Chiraz Trabelsi, \emph{Exponential
  ergodicity of the jump-diffusion {CIR} process}, Stochastics of environmental
  and financial economics---{C}entre of {A}dvanced {S}tudy, {O}slo, {N}orway,
  2014--2015, Springer Proc. Math. Stat., vol. 138, Springer, Cham, 2016,
  pp.~285--300. \MR{3451177}

\bibitem{MR3437080}
\bysame, \emph{Positive {H}arris recurrence and exponential ergodicity of the
  basic affine jump-diffusion}, Stoch. Anal. Appl. \textbf{34} (2016), no.~1,
  75--95. \MR{3437080}

\bibitem{MR1876169}
Olav Kallenberg, \emph{Foundations of modern probability}, second ed.,
  Probability and its Applications (New York), Springer-Verlag, New York, 2002.

\bibitem{MR2779872}
Martin Keller-Ressel, \emph{Moment explosions and long-term behavior of affine
  stochastic volatility models}, Math. Finance \textbf{21} (2011), no.~1,
  73--98. \MR{2779872}

\bibitem{MR2922631}
Martin Keller-Ressel and Aleksandar Mijatovi{\'c}, \emph{On the limit
  distributions of continuous-state branching processes with immigration},
  Stochastic Process. Appl. \textbf{122} (2012), no.~6, 2329--2345.
  \MR{2922631}

\bibitem{MR2390186}
Martin Keller-Ressel and Thomas Steiner, \emph{Yield curve shapes and the
  asymptotic short rate distribution in affine one-factor models}, Finance
  Stoch. \textbf{12} (2008), no.~2, 149--172. \MR{2390186}

\bibitem{MR3343292}
Zenghu Li and Chunhua Ma, \emph{Asymptotic properties of estimators in a stable
  {C}ox-{I}ngersoll-{R}oss model}, Stochastic Process. Appl. \textbf{125}
  (2015), no.~8, 3196--3233. \MR{3343292}

\bibitem{MR2509253}
Sean Meyn and Richard~L. Tweedie, \emph{Markov chains and stochastic
  stability}, second ed., Cambridge University Press, Cambridge, 2009, With a
  prologue by Peter W. Glynn. \MR{2509253}

\bibitem{MR1174380}
Sean~P. Meyn and R.~L. Tweedie, \emph{Stability of {M}arkovian processes. {I}.
  {C}riteria for discrete-time chains}, Adv. in Appl. Probab. \textbf{24}
  (1992), no.~3, 542--574. \MR{1174380 (93g:60143)}

\bibitem{MR1234294}
\bysame, \emph{Stability of {M}arkovian processes. {II}. {C}ontinuous-time
  processes and sampled chains}, Adv. in Appl. Probab. \textbf{25} (1993),
  no.~3, 487--517. \MR{1234294 (94g:60136)}

\bibitem{MR1234295}
\bysame, \emph{Stability of {M}arkovian processes. {III}. {F}oster-{L}yapunov
  criteria for continuous-time processes}, Adv. in Appl. Probab. \textbf{25}
  (1993), no.~3, 518--548. \MR{1234295 (94g:60137)}

\bibitem{MR1614256}
Ludger Overbeck, \emph{Estimation for continuous branching processes}, Scand.
  J. Statist. \textbf{25} (1998), no.~1, 111--126. \MR{1614256}

\bibitem{MR1455180}
Ludger Overbeck and Tobias Ryd{\'e}n, \emph{Estimation in the
  {C}ox-{I}ngersoll-{R}oss model}, Econometric Theory \textbf{13} (1997),
  no.~3, 430--461. \MR{1455180}

\bibitem{MR3185174}
Ken-iti Sato, \emph{L\'evy processes and infinitely divisible distributions},
  Cambridge Studies in Advanced Mathematics, vol.~68, Cambridge University
  Press, Cambridge, 2013, Translated from the 1990 Japanese original, Revised
  edition of the 1999 English translation. \MR{3185174}

\bibitem{S-R}
Rong Situ, \emph{Theory of stochastic differential equations with jumps and
  applications}, Springer, vol.~1, Springer-Verlag, 2010.

\bibitem{MR3235239}
Oldrich Vasicek, \emph{An equilibrium characterization of the term structure
  [reprint of {J}. {F}inanc. {E}con. {\bf 5} (1977), no. 2, 177--188]},
  Financial risk measurement and management, Internat. Lib. Crit. Writ. Econ.,
  vol. 267, Edward Elgar, Cheltenham, 2012, pp.~724--735. \MR{3235239}

\end{thebibliography}
%%
% requires a BiBTeX file sample.bib

\end{document}